\documentclass[11pt]{article}

\usepackage{amssymb}
\usepackage{amsmath}
\usepackage[utf8]{inputenc}
\usepackage[T1]{fontenc}
\usepackage{vmargin}
\usepackage{verbatim}
\usepackage[pdftex]{graphicx,color}
\graphicspath{{./Figures/}{.}}

\setpapersize{A4}
\setmarginsrb{30mm}{10mm}{30mm}{30mm}{10mm}{10mm}{10mm}{10mm}

\newcommand{\nbOne}{\mathchoice {\rm 1\mskip-4mu l} {\rm 1\mskip-4mu l}
{\rm 1\mskip-4.5mu l} {\rm 1\mskip-5mu l}}

\newcommand{\ieg}{\left[\hspace{-0.9ex}\left[\hspace{0.5ex}}
\newcommand{\ied}{\hspace{0.5ex} \right]\hspace{-0.9ex}\right]}
\newcommand{\argmax}{\operatornamewithlimits{argmax}}

\newcommand{\relmiddle}[1]{\mathrel{}\middle#1\mathrel{}}

\newcommand{\BlackBox}{\rule{1.5ex}{1.5ex}}  
\newenvironment{proof}{\par\noindent{\bf Proof\ }}{\hfill\BlackBox\\[2mm]}

\newtheorem{definition}{Definition}
\newtheorem{lemma}{Lemma}

\newtheorem{remark}{Remark}
\newtheorem{theorem}{Theorem}

\newtheorem{hypothesis}{Hypothesis}

\long\def\acks#1{\vskip 0.3in\noindent{\large\bf Acknowledgements}
\noindent #1}

\begin{document}
\date{\today}

\begin{titlepage}

\title{$L_p$-norm Sauer-Shelah Lemma for Margin Multi-category Classifiers}

\author{
    Yann Guermeur \\
    LORIA-CNRS \\
    Campus Scientifique, BP~239 \\
    54506 Vand\oe uvre-l\`es-Nancy Cedex, France \\
    (e-mail: {\tt Yann.Guermeur@loria.fr})
  }

\maketitle

\noindent{\bf Running Title}: 
Sauer-Shelah Lemma for Margin Multi-category Classifiers



\thispagestyle{empty}

\end{titlepage}

\begin{abstract}
In the framework of agnostic learning,
one of the main open problems of the theory of multi-category
pattern classification is the characterization of the way the complexity
varies with the number $C$ of categories. More precisely,
if the classifier is characterized only through
minimal learnability hypotheses,
then the optimal dependency on $C$ that an upper bound on
the probability of error should exhibit is unknown.
We consider margin classifiers.
They are based on classes of vector-valued functions
with one component function per category,
and the classes of component functions
are uniform Glivenko-Cantelli classes.
For these classifiers, an $L_p$-norm Sauer-Shelah lemma is established.
It is then used to derive guaranteed risks
in the $L_\infty$ and $L_2$-norms.
These bounds improve over the state-of-the-art ones with respect to
their dependency on $C$, which is sublinear.
\end{abstract}

\section{Introduction}
\label{sec:introduction}

During a long period, the theory of multi-category pattern classification
was considered as a topic of limited importance. Two connected reasons
can be put forward to explain this phenomenon. On the one hand,
the theory dedicated to dichotomies was making rapid strides, on the other hand,
decomposition methods were seen as efficient solutions to tackle polytomies.
An obvious drawback of this line of reasoning is to neglect the specificities
of the multi-category case, such as the dependency of the complexity of the task
on the number $C$ of categories. In recent years, several studies addressed
this question, by deriving upper bounds on the probability of error
of multi-category classifiers, especially margin ones. However, most of
these {\em guaranteed risks}
were dedicated to specific families of classifiers, let them be kernel machines
\cite{Zha04,LeiDogBinKlo15}, neural networks \cite{AntBar99},
decision trees \cite{KuzMohSye14}
or nearest neighbors classifiers \cite{KonWei14}.
This article deals with margin classifiers.
They are based on classes of vector-valued functions
with one component function per category,
and the classes of component functions
are uniform Glivenko-Cantelli classes.
For these classifiers, an $L_p$-norm Sauer-Shelah lemma is established.
It is then used to derive guaranteed risks
in the $L_\infty$ and $L_2$-norms.
These bounds improve over the state-of-the-art ones with respect to
their dependency on $C$, which is sublinear.
Thus, they pave the way for the characterization of
the optimal dependency on $C$ that could be obtained in the framework
of agnostic learning, under minimal learnability/measurability
hypotheses regarding the classes of functions involved.

The organization of the paper is as follows.
Section~\ref{sec:margin-multi-category-classifiers}
deals with the theoretical framework
and the margin multi-category classifiers.
Section~\ref{sec:L_p-norm Sauer-Shelah Lemma} is devoted to the derivation
of the $L_p$-norm Sauer-Shelah lemma.
The bound based on the $L_\infty$-norm and that based on the $L_2$-norm
are respectively established in
Section~\ref{sec:L_infty-bound} and Section~\ref{sec:L_2-bound}.
At last, we draw conclusions and outline our ongoing research
in Section~\ref{sec:conclusions}. To make reading easier, 
basic results from the literature and technical lemmas
have been gathered in appendix.

\section{Margin multi-category classifiers}
\label{sec:margin-multi-category-classifiers}

The theoretical framework for the margin multi-category classifiers
has been introduced in \cite{Gue07b}. It is summarized below.

\subsection{Theoretical framework}

We consider the case of $C$-category pattern classification problems
\cite{DevGyoLug96} with $C \in \mathbb{N} \setminus \ieg 0, 2 \ied$.
Each object is represented by its description $x \in \mathcal{X}$ and the
set $\mathcal{Y}$ of the categories $y$ can be identified with the
set of indices of the categories: $\ieg 1, C \ied$.
We assume that $\left ( \mathcal{X}, \mathcal{A}_{\mathcal{X}} \right )$ and
$\left ( \mathcal{Y}, \mathcal{A}_{\mathcal{Y}} \right )$ 
are measurable spaces
and denote by $\mathcal{A}_\mathcal{X} \otimes \mathcal{A}_\mathcal{Y}$ 
the tensor-product sigma algebra
on the Cartesian product $ \mathcal{X} \times \mathcal{Y}$.
We make the hypothesis that the link between
descriptions and categories can be characterized by an unknown probability
measure $P$ on the measurable space
$\left ( \mathcal{X} \times \mathcal{Y}, 
\mathcal{A}_{\mathcal{X}} \otimes \mathcal{A}_{\mathcal{Y}} \right )$.
Let $Z = \left ( X,Y \right )$ be a random pair
with values in $\mathcal{Z} = \mathcal{X} \times \mathcal{Y}$,
distributed according to $P$.
The single knowledge source on $P$ available is an $m$-sample
$\mathbf{Z}_m = \left( Z_i  \right)_{1 \leqslant i \leqslant m} =
\left( \left ( X_i, Y_ i \right) \right)_{1 \leqslant i \leqslant m}$
made up of independent copies of $Z$ (in short $\mathbf{Z}_m \sim P^m$).
The theoretical framework is thus that of {\em agnostic learning} 
\cite{KeaSchSel94}. To simplify reasoning, in the sequel, 
the hypothesis $m > C$ is made.

We add an hypothesis to that framework: the fact that the classifiers
considered are based on classes of vector-valued functions 
with one component function per category,
and the classes of component functions are {\em uniform Glivenko-Cantelli}.
The definition of this property calls for the introduction of
an intermediate definition.

\begin{definition}[Empirical probability measure]
\label{def:empirical-probability-measure}
Let $\left ( \mathcal{T}, \mathcal{A}_{\mathcal{T}} \right )$
be a measurable space and 
let $T$ be a random variable with values in $\mathcal{T}$,
distributed according to a probability measure $P_T$
on $\left ( \mathcal{T}, \mathcal{A}_{\mathcal{T}} \right )$.
For $n \in \mathbb{N}^*$,
let $\mathbf{T}_n = \left( T_i  \right)_{1 \leqslant i \leqslant n}$
be an $n$-sample made up of independent copies of $T$.
The empirical measure supported on this sample, $P_{\mathbf{T}_n}$, is given by
$$
P_{\mathbf{T}_n} = \frac{1}{n} \sum_{i=1}^n \delta_{T_i},
$$
where $\delta_{T_i}$ denotes the Dirac measure centered on $T_i$.
\end{definition}

\begin{definition}[Uniform Glivenko-Cantelli class \cite{DudGinZin91}]
Let the probability measures $P_T$ and $P_{\mathbf{T}_n}$ be defined as 
in Definition~\ref{def:empirical-probability-measure}.
Let $\mathcal{F}$ be a class of measurable functions on $\mathcal{T}$.
Then $\mathcal{F}$ is a {\em uniform Glivenko-Cantelli class}
if for every $\epsilon \in \mathbb{R}_+^*$,
$$
\lim_{n \longrightarrow + \infty} \sup_{P_T} \;
\mathbb{P} \left ( \sup_{n' \geqslant n} \sup_{f \in \mathcal{F}}
\left |
\mathbb{E}_{T' \sim P_{\mathbf{T}_{n'}}} \left  [ f \left ( T' \right ) \right ]
- \mathbb{E}_{T \sim P_T} \left  [ f \left ( T \right ) \right ]
\right | > \epsilon \right ) = 0,
$$
where $\mathbb{P}$ denotes the infinite product measure $P_T^{\infty}$.
\end{definition}
Henceforth, we shall refer to uniform Glivenko-Cantelli classes
by the abbreviation {\em GC classes}. GC classes must be uniformly bounded
up to additive constants 
(see for instance Proposition~4 in \cite{DudGinZin91}).
For notational convenience, we replace this property by a stronger one:
the vector-valued functions take their values in a hypercube of
$\mathbb{R}^C$. The definition of a margin multi-category classifier
is thus the following one.

\begin{definition}[Margin multi-category classifiers]
\label{def:margin-multi-category-classifiers}
Let $\mathcal{G} = \prod_{k=1}^C \mathcal{G}_k$ be a class of
functions from $\mathcal{X}$ into
$\left [-M_{\mathcal{G}}, M_{\mathcal{G}} \right ]^{C}$
with $M_{\mathcal{G}} \in \left [ 1, +\infty \right )$.
The classes $\mathcal{G}_k$ of component functions are supposed to be 
GC classes.
For each function $g = \left ( g_{k} \right )_{1 \leqslant k \leqslant C}
\in \mathcal{G}$, a {\em margin multi-category classifier} on
$\mathcal{X}$ is obtained by application of 
the operator $\text{dr}$ 
from $\mathcal{G}$ into 
$\left ( \mathcal{Y} \bigcup \left \{ * \right \} \right )^{\mathcal{X}}$
named {\em decision rule} and defined as follows:
$$
\forall x \in \mathcal{X}, \;\;
\begin{cases}
\left | \argmax_{1 \leqslant k \leqslant C} g_k \left ( x \right ) \right | = 1
\Longrightarrow \text{dr}_g \left ( x \right ) =
\argmax_{1 \leqslant k \leqslant C} g_k \left ( x \right ) \\
\left | \argmax_{1 \leqslant k \leqslant C} g_k \left ( x \right ) \right | > 1
\Longrightarrow \text{dr}_g \left ( x \right ) = *
\end{cases}
$$
where $\left | \cdot \right |$ returns the cardinality
of its argument and $*$ stands for a dummy category.
\end{definition}
In words, $\text{dr}_g$ returns either the index of
the component function whose value is the highest, or the dummy category $*$
in case of ex \ae quo. In the case when the $g_k \left ( x \right )$ 
are class posterior
probability estimates, then $\text{dr}$ is simply Bayes' estimated
decision rule \cite{RicLip91}.
The qualifier {\em margin} refers to the fact that the generalization
capabilities of such classifiers can be characterized by means of the values
taken by the differences of the corresponding component functions.
The use of the dummy category to avoid breaking ties
is not central to the theory. Its main advantage rests in the fact that
it keeps the reasoning and formulas as simple as possible. 

With this definition at hand, the aim of the {\em learning process}
is to minimize over $\mathcal{G}$ the {\em probability of error}
$P \left ( \text{dr}_g \left ( X \right ) \neq Y \right )$.
This probability can be reformulated in a handy way
thanks to the introduction of additional functions.

\begin{definition}[Class of functions $\mathcal{F}_{\mathcal{G}}$]
\label{def:class-of-transformed-functions}
Let $\mathcal{G}$ be a class of functions satisfying
Definition~\ref{def:margin-multi-category-classifiers}.
For all $g \in \mathcal{G}$, the function
$f_g$ from $\mathcal{X} \times \ieg 1, C \ied$ 
into $\left [-M_{\mathcal{G}}, M_{\mathcal{G}} \right ]$ is defined by:
$$
\forall \left ( x, k \right ) \in \mathcal{X} \times \ieg 1, C \ied, \;\;
f_g \left ( x, k \right ) = \frac{1}{2} \left ( g_k \left ( x \right ) -
\max_{l \neq k} g_l \left ( x \right ) \right ).
$$
Then, the class $\mathcal{F}_{\mathcal{G}}$ is defined as follows:
$$
\mathcal{F}_{\mathcal{G}} = \left \{ f_g: \; g \in \mathcal{G} \right \}.
$$
\end{definition}

\begin{definition}[Expected risk $L$]
Let $\mathcal{G}$ be a class of functions satisfying
Definition~\ref{def:margin-multi-category-classifiers}
and let $\phi$ be the standard indicator loss function given by:
$$
\forall t \in \mathbb{R}, \;\;
\phi \left ( t \right ) = \nbOne_{\left \{ t \leqslant 0 \right \}}.
$$
The {\em expected risk} of any function $g \in \mathcal{G}$,
$L \left ( g \right )$, is given by:
$$
L \left ( g \right )
= \mathbb{E}_{\left ( X, Y \right ) \sim P}
\left [ \phi \circ f_g \left ( X, Y \right ) \right ]
= P \left ( \text{dr}_g \left ( X \right ) \neq Y \right ).
$$
Its {\em empirical risk} measured on the $m$-sample $\mathbf{Z}_m$ is:
$$
L_m \left ( g \right ) =
\mathbb{E}_{Z' \sim P_m}
\left [ \phi \circ f_g \left ( Z' \right ) \right ] =
\frac{1}{m} \sum_{i=1}^m
\phi \circ f_g \left ( Z_i \right ).
$$
\end{definition}
In order to take benefit from the fact that the classifiers of interest
are margin ones, the sample-based estimate of performance which is
actually used (involved in the different guaranteed risks)
is obtained by substituting to $\phi$
a (dominating) margin loss/cost function.
In this study, the definition used for those functions is the following one.

\begin{definition}[Margin loss functions]
\label{def:margin-loss-functions}
A class of {\em margin loss functions} $\phi_{\gamma}$
parameterized by $\gamma \in \left ( 0, 1 \right ]$ is a class
of nonincreasing functions from $\mathbb{R}$ into $\left [ 0, 1 \right ]$
satisfying:
\begin{enumerate}
\item $\forall \gamma \in \left ( 0, 1 \right ], \;\; 
       \phi_{\gamma} \left ( 0 \right ) = 1 
       \wedge \phi_{\gamma} \left ( \gamma \right ) = 0$;
\item $\forall \left ( \gamma, \gamma' \right ) \in \left ( 0, 1 \right ]^2,
       \;\; \gamma < \gamma' \Longrightarrow \forall t 
       \in \left ( 0, \gamma \right ), \;\;
       \phi_{\gamma} \left ( t \right )
       \leqslant \phi_{\gamma'} \left ( t \right )$.
\end{enumerate}
\end{definition}

\begin{remark}
The qualifier {\em dominating} is appropriate since we have for all
$\left ( \gamma, t \right ) \in \left ( 0, 1 \right ] \times \mathbb{R}$,
$\phi_{\gamma} \left ( t \right ) \geqslant \phi \left ( t \right )$.
The second property is especially useful to derive guaranteed risks
holding uniformly for all values of $\gamma$.
This can be achieved by means of Proposition~8 in \cite{Bar98}.
It is noteworthy that these losses are not convex. They can even be
discontinuous (whereas the definition used by Koltchinskii and Panchenko
in \cite{KolPan02} (Section~2) includes the Lipschitz property). 
\end{remark}
A risk obtained by substituting to $\phi$ a function
$\phi_{\gamma}$ is named a margin risk.

\begin{definition}[Margin risk $L_{\gamma}$]
\label{def:margin_risk}
Let $\mathcal{G}$ be a class of functions
satisfying Definition~\ref{def:margin-multi-category-classifiers}.
For every (ordered) pair $\left ( g, \gamma \right )
\in \mathcal{G} \times \left (0, 1 \right ]$,
the {\em risk with margin $\gamma$} of $g$,  $L_{\gamma} \left ( g \right )$,
is defined as:
$$
L_{\gamma} \left ( g \right ) = \mathbb{E}_{Z \sim P} 
\left [ \phi_{\gamma} \circ f_g \left ( Z \right ) \right ].
$$
$L_{\gamma, m} \left ( g \right )$
designates the corresponding empirical risk,
measured on the $m$-sample $\mathbf{Z}_m$:
$$
L_{\gamma, m} \left ( g \right ) = 
\mathbb{E}_{Z' \sim P_m} 
\left [ \phi_{\gamma} \circ f_g \left ( Z' \right ) \right ] =
\frac{1}{m} \sum_{i=1}^m
\phi_{\gamma} \circ f_g \left ( Z_i \right ).
$$
\end{definition}
Taking our inspiration from \cite{Bar98},
we use margin loss functions in combination
with a piecewise-linear squashing function.
In short, the idea is to restrict the available information 
to what is relevant for the assessment of the prediction accuracy
(the value of the margin loss is not affected), 
so as to optimize the way the introduction
of the margin parameter $\gamma$ is taken into account.

\begin{definition}[Piecewise-linear squashing function $\pi_{\gamma}$]
\label{def:piecewise-linear-squashing-function}
For $\gamma \in \left (0, 1 \right ]$,
the {\em piecewise-linear squashing function} $\pi_{\gamma}$ is defined by:
$$
\forall t \in \mathbb{R}, \;\;
\pi_{\gamma} \left ( t \right ) =
t \nbOne_{\left \{ t \in \left ( 0, \gamma \right ] \right \}}
+ \gamma \nbOne_{\left \{ t > \gamma \right \}}.
$$
\end{definition}
This definition actually satisfies the aforementioned specification
since we have:
$$
\forall \gamma \in \left ( 0, 1 \right ], \;\;
\phi_{\gamma} \circ \pi_{\gamma} = \phi_{\gamma}.
$$

\begin{definition}[Class of functions $\mathcal{F}_{\mathcal{G}, \gamma}$]
\label{def:class-of-regret-functions}
Let $\mathcal{G}$ be a class of functions
satisfying Definition~\ref{def:margin-multi-category-classifiers} and
$\mathcal{F}_{\mathcal{G}}$ the class of functions deduced from
$\mathcal{G}$ according to Definition~\ref{def:class-of-transformed-functions}.
For every pair $\left ( g, \gamma \right )
\in \mathcal{G} \times \left (0, 1 \right ]$,
the function
$f_{g, \gamma}$ from $\mathcal{X} \times \ieg 1, C \ied$
into $\left [ 0, \gamma \right ]$ is defined by:
$$
f_{g, \gamma} = \pi_{\gamma} \circ f_g.
$$
Then, the class $\mathcal{F}_{\mathcal{G}, \gamma}$ is defined as follows:
$$
\mathcal{F}_{\mathcal{G}, \gamma} = 
\left \{ f_{g, \gamma}: \; g \in \mathcal{G} \right \}.
$$
\end{definition}

\subsection{Scale-sensitive capacity measures}
\label{sec:capacity-measures}

The guaranteed risks are ordinarily
obtained in several main steps,
corresponding to a basic supremum inequality and successive
upper bounds on the capacity measure it involves, 
each of which corresponds to a change of capacity measure.
Although the measures which are central to this study are covering numbers, 
we start by giving the definition of the Rademacher complexity
since it is the measure appearing
first in the case of the $L_2$-norm.
For $n \in \mathbb{N}^*$, a Rademacher sequence $\boldsymbol{\sigma}_n$
is a sequence $\left ( \sigma_i \right )_{1 \leqslant i \leqslant n}$
of independent random signs, i.e., independent and identically distributed
random variables taking the values $-1$ and $1$ with probability
$\frac{1}{2}$ (symmetric Bernoulli or Rademacher random variables).

\begin{definition}[Rademacher complexity]
\label{def:Rademacher-complexity}
Let $\left ( \mathcal{T}, \mathcal{A}_{\mathcal{T}} \right )$
be a measurable space and
let $T$ be a random variable with values in $\mathcal{T}$,
distributed according to a probability measure $P_T$
on $\left ( \mathcal{T}, \mathcal{A}_{\mathcal{T}} \right )$.
For $n \in \mathbb{N}^*$,
let $\mathbf{T}_n = \left( T_i  \right)_{1 \leqslant i \leqslant n}$
be an $n$-sample made up of independent copies of $T$ and let
$\boldsymbol{\sigma}_n = \left ( \sigma_i \right )_{1 \leqslant i \leqslant n}$
be a Rademacher sequence.
Let $\mathcal{F}$ be a class of real-valued functions with domain $\mathcal{T}$.
The {\em empirical Rademacher complexity} of $\mathcal{F}$ is
$$
\hat{R}_n \left ( \mathcal{F} \right ) = 
\mathbb{E}_{\boldsymbol{\sigma}_n}
\left [ \sup_{f \in \mathcal{F}} \frac{1}{n}
\sum_{i=1}^n \sigma_i f \left ( T_i \right )
\relmiddle| \mathbf{T}_n \right ].
$$
The {\em Rademacher complexity} of $\mathcal{F}$ is
$$
R_n \left ( \mathcal{F} \right ) 
= \mathbb{E}_{\mathbf{T}_n} 
\left [ \hat{R}_n \left ( \mathcal{F} \right ) \right ]
= \mathbb{E}_{\mathbf{T}_n \boldsymbol{\sigma}_n} \left [
\sup_{f \in \mathcal{F}} \frac{1}{n}
\sum_{i=1}^n \sigma_i f \left ( T_i \right ) \right ].
$$
\end{definition}

\begin{remark}
The fact that the functional classes $\mathcal{F}$ of interest can be
uncountable calls for a specification. We make use of the standard convention
(see for instance Formula~(0.2) in \cite{Tal05}).
Let $\left ( T_s \right )_{s \in \mathcal{S}}$
be a stochastic process. Then,
$$
\mathbb{E} \left [ \sup_{s \in \mathcal{S}} T_s \right ] =
\sup_{ \left \{ \bar{\mathcal{S}} \subset \mathcal{S}: \;
\left | \bar{\mathcal{S}} \right | < +\infty \right \} }
\mathbb{E} \left [ \max_{s \in \bar{\mathcal{S}}} T_s \right ].
$$
\end{remark}
The concept of covering number ($\epsilon$-entropy), 
as well as the underlying
concepts of $\epsilon$-cover and $\epsilon$-net,
can be traced back to \cite{KolTih61}.

\begin{definition}[$\epsilon$-cover, $\epsilon$-net, covering numbers,
and $\epsilon$-entropy]
\label{def:covering-numbers}
Let $ \left ( E, \rho \right )$ 
be a pseudo-metric space, $E' \subset E$ and
$\epsilon \in \mathbb{R}_+^*$.
An {\em $\epsilon$-cover} of $E'$ is a coverage of $E'$
with open balls of radius $\epsilon$
the centers of which belong to $E$. These centers form an {\em $\epsilon$-net}
of $E'$. A {\em proper $\epsilon$-net} of $E'$ is an
{\em $\epsilon$-net} of $E'$ included in $E'$.
If $E'$ has an $\epsilon$-net of finite cardinality, then its
{\em covering number}
$\mathcal{N} \left ( \epsilon, E', \rho \right )$ is the smallest cardinality
of its $\epsilon$-nets:
$$
\mathcal{N} \left ( \epsilon, E', \rho \right ) =
\min \left \{ \left | E'' \right |: \; \left ( E'' \subset E \right ) \wedge 
\left ( \forall e \in E', \; \rho \left ( e, E'' \right ) < \epsilon 
\right ) \right \}.
$$
If there is no such finite net, then
the covering number is defined to be infinite.
The corresponding logarithm,
$\log_2 \left ( \mathcal{N} \left ( \epsilon, E', \rho \right ) \right )$,
is called the {\em minimal $\epsilon$-entropy} of $E'$, or simply
the {\em $\epsilon$-entropy} of $E'$.
$\mathcal{N}^{(p)} \left ( \epsilon, E', \rho \right )$ 
will designate a covering number
of $E'$ obtained by considering proper $\epsilon$-nets only.
In the finite case, we have thus:
$$
\mathcal{N}^{(p)} \left ( \epsilon, E', \rho \right ) =
\min \left \{ \left | E'' \right |: \; \left ( E'' \subset E' \right ) \wedge 
\left ( \forall e \in E', \; \rho \left ( e, E'' \right ) < \epsilon 
\right ) \right \}.
$$
\end{definition}
There is a close connection between covering and packing properties
of bounded subsets in pseudo-metric spaces.

\begin{definition}[$\epsilon$-separation and packing numbers \cite{KolTih61}]
Let $ \left ( E, \rho \right )$ be a pseudo-metric space and 
$\epsilon \in \mathbb{R}_+^*$.
A set $E' \subset E$ is
{\em $\epsilon$-separated} if, for any distinct points $e$ and $e'$ in $E'$,
$\rho \left ( e, e' \right) \geqslant \epsilon$.
The {\em $\epsilon$-packing number} of $E'' \subset E$,
$\mathcal{M} \left ( \epsilon, E'', \rho \right )$,
is the maximal cardinality of an $\epsilon$-separated subset of $E''$,
if such maximum exists. Otherwise, the $\epsilon$-packing number of $E''$
is defined to be infinite.
\end{definition}
In this study, the functional classes met are endowed with
empirical (pseudo-)metrics derived from the $L_p$-norm.

\begin{definition}[Pseudo-distance $d_{p, \mathbf{t}_n}$]
\label{def:pseudo-metric-d_p}
Let $\mathcal{F}$ be a class of real-valued functions on
$\mathcal{T}$. For $n \in \mathbb{N}^*$, let
$\mathbf{t}_n = 
\left ( t_i \right )_{1 \leqslant i \leqslant n} \in \mathcal{T}^n$. Then,
$$
\forall p \in \mathbb{N}^*,
\forall \left ( f, f' \right ) \in \mathcal{F}^2, \;\;
d_{p, \mathbf{t}_n}  \left ( f, f' \right ) 
= \left \| f - f' \right \|_{L_p \left ( \mu_{\mathbf{t}_n} \right )}
= \left ( \frac{1}{n} \sum_{i=1}^n
\left | f \left ( t_i \right ) -
f' \left ( t_i \right ) \right |^p
\right )^{\frac{1}{p}}
$$
and
$$
\forall \left ( f, f' \right ) \in \mathcal{F}^2, \;\;
d_{\infty, \mathbf{t}_n}  \left ( f, f' \right )
= \left \| f - f' \right \|_{L_{\infty} \left ( \mu_{\mathbf{t}_n} \right )}
= \max_{1 \leqslant i \leqslant n}
\left | f \left ( t_i \right ) - f' \left ( t_i \right ) \right |,
$$
where $\mu_{\mathbf{t}_n}$ denotes the uniform (counting) probability measure
on $\left \{ t_i: 1 \leqslant i \leqslant n \right \}$.
\end{definition}

\begin{definition}[Uniform covering numbers \cite{WilSmoSch01}
and uniform packing numbers \cite{Bar98}]
\label{def:uniform-covering-numbers}
Let  $\mathcal{F}$ be a class of real-valued functions on
$\mathcal{T}$ and $\bar{\mathcal{F}} \subset \mathcal{F}$.
For $p \in \mathbb{N}^* \bigcup \left \{ +\infty \right \}$,
$\epsilon \in \mathbb{R}_+^*$, and $n \in \mathbb{N}^*$,
the {\em uniform covering number} 
$\mathcal{N}_p \left ( \epsilon, \bar{\mathcal{F}}, n \right )$ 
and the {\em uniform packing number}
$\mathcal{M}_p \left ( \epsilon, \bar{\mathcal{F}}, n \right )$ 
are defined as follows:
$$
\begin{cases}
\mathcal{N}_p \left ( \epsilon, \bar{\mathcal{F}}, n \right ) =
\sup_{\mathbf{t}_n \in \mathcal{T}^n}
\mathcal{N} \left ( \epsilon, \bar{\mathcal{F}}, 
d_{p, \mathbf{t}_n} \right ) \\
\mathcal{M}_p \left ( \epsilon, \bar{\mathcal{F}}, n \right ) =
\sup_{\mathbf{t}_n \in \mathcal{T}^n}
\mathcal{M} \left ( \epsilon, \bar{\mathcal{F}}, 
d_{p, \mathbf{t}_n} \right )
\end{cases}.
$$
We define accordingly 
$\mathcal{N}_p^{(p)} \left ( \epsilon, \bar{\mathcal{F}}, n \right )$ as:
$$
\mathcal{N}_p^{(p)} \left ( \epsilon, \bar{\mathcal{F}}, n \right ) =
\sup_{\mathbf{t}_n \in \mathcal{T}^n}
\mathcal{N}^{(p)} \left ( \epsilon, \bar{\mathcal{F}}, 
d_{p, \mathbf{t}_n} \right ).
$$
\end{definition}
Our Sauer-Shelah lemma relates covering/packing numbers to
a scale-sensitive generalization of 
the Vapnik-Chervonenkis (VC) dimension \cite{VapChe71}:
the fat-shattering dimension \cite{KeaSch94} also known as 
the $\gamma$-dimension.

\begin{definition}[Fat-shattering dimension \cite{KeaSch94}]
\label{definition:fat_shattering_dim}
Let $\mathcal{F}$ be a class of functions from $\mathcal{T}$
into $\mathcal{S} \subset \mathbb{R}$.
For $\gamma \in \mathbb{R}_+^*$,
a subset $s_{\mathcal{T}^n} = \left \{ t_i: 1 \leqslant i \leqslant n \right \}$
of $\mathcal{T}$
is said to be {\em ${\gamma}$-shattered} by $\mathcal{F}$ if
there is a vector 
$\mathbf{b}_n =  \left ( b_i \right )_{1 \leqslant i \leqslant n}
\in \mathcal{S}^n$ 
such that, for every vector
$\mathbf{l}_n = \left ( l_i \right )_{1 \leqslant i \leqslant n}
\in \left \{ -1, 1 \right \}^n$, there is a function
$f_{\mathbf{l}_n} \in \mathcal{F}$ satisfying
$$
\forall i \in \ieg 1, n \ied, \;\;
l_i \left( f_{\mathbf{l}_n} \left ( t_i \right ) - b_i \right) \geqslant \gamma.
$$
The vector $\mathbf{b}_n$ is called a {\em witness} to 
the ${\gamma}$-shattering.
The {\em fat-shattering dimension with margin $\gamma$} of the class
$\mathcal{F}$, $\gamma\mbox{-dim} \left ( \mathcal{F} \right )$,
is the maximal cardinality of a subset of $\mathcal{T}$
${\gamma}$-shattered by $\mathcal{F}$,
if such maximum exists.
Otherwise,
$\mathcal{F}$ is said to have infinite 
fat-shattering dimension with margin $\gamma$.
\end{definition}

\begin{remark}
With the introduction of the set $\mathcal{S}$ (and the constraint
$\mathbf{b}_n \in \mathcal{S}^n$) 
in Definition~\ref{definition:fat_shattering_dim},
there is no need to make use of the strong dimension
(Definition~3.1 in \cite{AloBenCesHau97}).
A difference with the definition
used in \cite{MenVer03} regards the concept of shattering.
As most of the authors (see for instance \cite{AloBenCesHau97}),
we do not adopt the convention consisting in
considering that the empty set can be shattered. Using the
terminology of Mendelson and Vershynin
(see Section~2.2 in \cite{MenVer03}), the {\em trivial center}
is not involved in our computations.
\end{remark}
Each of the generalized Sauer-Shelah lemmas in the literature is based on
a main combinatorial result
that involves a class of functions whose domain
and codomain are finite sets.
The first property is simply obtained by application of a restriction
of the domain to the data at hand.
As for the finiteness of the codomain, if needed,
it is obtained by application of a discretization operator.
The present study makes use of the following one, already employed, 
for instance, in \cite{BarLon95}.

\begin{definition}[$\eta$-discretization operator]
Let $\mathcal{F}$ be a class of functions from $\mathcal{T}$
into $\left [ -M_{\mathcal{F}}, M_{\mathcal{F}} \right ]$
with $M_{\mathcal{F}} \in \mathbb{R}_+^*$.
For $\eta \in \mathbb{R}_+^*$, define the {\em $\eta$-discretization}
as an operator on $\mathcal{F}$ such that:
$$
\begin{array}{l l l l}
\left ( \cdot \right )^{\left ( \eta \right )}:
& \mathcal{F} & \longrightarrow & \mathcal{F}^{\left ( \eta \right )} \\
& f & \mapsto & f^{\left ( \eta \right )}
\end{array}
$$
$$
\forall t \in \mathcal{T}, \;\;
f^{\left ( \eta \right )} \left ( t \right ) =
\eta \left \lfloor \frac{f \left ( t \right )
+ M_{\mathcal{F}}}{\eta} \right \rfloor
$$
where the floor function $\left \lfloor \cdot \right \rfloor$ is defined by:
$$
\forall u \in \mathbb{R}, \;\; \left \lfloor u \right \rfloor =
\max \left \{ j \in \mathbb{Z}: \; j \leqslant u \right \}.
$$
\end{definition}
The finiteness of all the capacity measures considered in the sequel is ensured.
Precisely, Theorem~2.5 in \cite{AloBenCesHau97} 
(see also Theorem~2.4 in \cite{Men02}) tells us that
the fat-shattering dimension of a GC class is finite
for every positive value of $\gamma$,
and a corollary of the generalized Sauer-Shelah lemma
is that the finiteness of this dimension
implies the total boundedness.

\section{$L_p$-norm Sauer-Shelah Lemma}
\label{sec:L_p-norm Sauer-Shelah Lemma}

Our master lemma is made up of two partial results.
The first one, the {\em decomposition lemma},
relates the covering numbers of 
$\mathcal{F}_{\mathcal{G}, \gamma}$
to those of the classes of component functions $\mathcal{G}_k$.
The second one is the actual generalized Sauer-Shelah lemma.

\subsection{Master lemma}

\begin{lemma}[Decomposition lemma]
\label{lemma:from-multivariate-to-univariate-L_p}
Let $\mathcal{G}$ be a class of functions
satisfying Definition~\ref{def:margin-multi-category-classifiers}
and $\mathcal{F}_{\mathcal{G}}$ the class of functions deduced
from $\mathcal{G}$ according to
Definition~\ref{def:class-of-transformed-functions}.
For $\gamma \in \left ( 0, 1 \right ]$,
let $\mathcal{F}_{\mathcal{G}, \gamma}$ be the class of functions deduced
from $\mathcal{G}$ according to
Definition~\ref{def:class-of-regret-functions}.
Then, for $\epsilon \in \mathbb{R}_+^*$, 
$m \in \mathbb{N}^*$,
and $\mathbf{z}_m = \left ( \left ( x_i, y_i \right ) 
\right )_{1 \leqslant i \leqslant m} =
\left ( z_i \right )_{1 \leqslant i \leqslant m}$,

\begin{equation}
\label{eq:from-multivariate-to-univariate-L_p}
\forall p \in \mathbb{N}^* \bigcup \left \{ +\infty \right \}, \;\;
\mathcal{N}^{(p)} \left ( \epsilon, 
\mathcal{F}_{\mathcal{G}, \gamma}, d_{p, \mathbf{z}_m} \right ) \leqslant
\mathcal{N}^{(p)} \left ( \epsilon, 
\mathcal{F}_{\mathcal{G}}, d_{p, \mathbf{z}_m} \right ) \leqslant
\prod_{k=1}^C \mathcal{N}^{(p)} 
\left ( \frac{\epsilon}{C^{\frac{1}{p}}}, 
\mathcal{G}_k, d_{p, \mathbf{x}_m} \right ),
\end{equation}
where $\mathbf{x}_m = \left ( x_i \right )_{1 \leqslant i \leqslant m}$.
\end{lemma}

\begin{proof}
The left-hand side inequality in
Formula~(\ref{eq:from-multivariate-to-univariate-L_p})
is trivially true for $\epsilon > \gamma$.
Otherwise, it is a direct consequence of
the $1$-Lipschitz property of the function $\pi_{\gamma}$.
Similarly, the proof of the right-hand side inequality
is nontrivial only for $\epsilon \leqslant 2 M_{\mathcal{G}}$.
We first derive it for a finite value of $p$ only.
For every function $g = \left ( g_k \right )_{1 \leqslant k \leqslant C}
\in \mathcal{G}$ and every element $z = \left ( x, y \right )
\in \mathcal{X} \times \ieg 1, C \ied$, let
$k \left ( g, z \right ) \in \ieg 1, C \ied \setminus \left \{ y \right \}$
be an index of category such that
$f_g \left ( z \right ) = \frac{1}{2} \left ( 
g_{y} \left ( x \right ) - g_{k \left ( g, z \right )} 
\left ( x \right ) \right )$.
For all $k \in \ieg 1, C \ied$, let $\bar{\mathcal{G}}_k$ be a proper
$\epsilon$-net of $\mathcal{G}_k$ with respect to the pseudo-metric
$d_{p, \mathbf{x}_m}$ such that $\bar{\mathcal{G}}_k$
is of cardinality $\mathcal{N}^{(p)} \left ( \epsilon, \mathcal{G}_k, 
d_{p, \mathbf{x}_m} \right )$.
By construction, the cardinality of the class of functions 
$\bar{\mathcal{G}} = \prod_{k=1}^C \bar{\mathcal{G}}_k$ is
$\prod_{k=1}^C \mathcal{N}^{(p)} \left ( \epsilon, \mathcal{G}_k, 
d_{p, \mathbf{x}_m} \right )$,
and for every function 
$g = \left ( g_k \right )_{1 \leqslant k \leqslant C} \in \mathcal{G}$,
there exists a function
$\bar{g} = \left ( \bar{g}_k \right )_{1 \leqslant k \leqslant C}
\in \bar{\mathcal{G}}$ such that:
\begin{equation}
\label{eq:definition-of-bar-g}
\forall k \in \ieg 1, C \ied, \;\; 
d_{p, \mathbf{x}_m} \left ( g_k, \bar{g}_k \right ) < \epsilon.
\end{equation}
By definition of the empirical pseudo-metric,
for every $k \in \ieg 1, C \ied$ and every function $g_k \in \mathcal{G}_k$,
\begin{align}
d_{p, \mathbf{x}_m} \left ( g_k, \bar{g}_k \right ) < \epsilon
& \Longleftrightarrow \;
\left ( \frac{1}{m} \sum_{i=1}^m
\left | g_k \left ( x_i \right ) -
\bar{g}_k \left ( x_i \right ) \right |^p
\right )^{\frac{1}{p}} < \epsilon \nonumber \\
& \Longrightarrow \;
\forall i \in \ieg 1, m \ied, \;\;
\left | g_k \left ( x_i \right ) - \bar{g}_k \left ( x_i \right ) \right | 
< m^{\frac{1}{p}} \epsilon \nonumber \\
\label{eq:def_theta_ik}
& \Longrightarrow \;
\left ( \left | g_k \left ( x_i \right ) - 
\bar{g}_k \left ( x_i \right ) \right | 
\right )_{1 \leqslant i \leqslant m}  = m^{\frac{1}{p}} \epsilon
\left ( \theta_{ki} \right )_{1 \leqslant i \leqslant m}
\end{align}
where $\left ( \theta_{ki} 
\right )_{1 \leqslant i \leqslant m} \in \left [ 0, 1 \right )^m$.
Furthermore, if $g_{k \left ( g, z_i \right )} \left ( x_i \right )
\geqslant \bar{g}_{k \left ( \bar{g}, z_i \right )} \left ( x_i \right )$,
then
\begin{align*}
\left | g_{k \left ( g, z_i \right )} \left ( x_i \right ) -
\bar{g}_{k \left ( \bar{g}, z_i \right )} \left ( x_i \right )
\right |
& = \; 
g_{k \left ( g, z_i \right )} \left ( x_i \right ) -
\bar{g}_{k \left ( \bar{g}, z_i \right )} \left ( x_i \right ) \\
& \leqslant \; 
g_{k \left ( g, z_i \right )} \left ( x_i \right ) -
\bar{g}_{k \left ( g, z_i \right )} \left ( x_i \right ) \\
& \leqslant \; 
\left |  g_{k \left ( g, z_i \right )} \left ( x_i \right ) -
\bar{g}_{k \left ( g, z_i \right )} \left ( x_i \right ) \right | \\
& \leqslant \; 
\theta_{k \left ( g, z_i \right ) i} m^{\frac{1}{p}} \epsilon.
\end{align*}
Symmetrically, $g_{k \left ( g, z_i \right )} \left ( x_i \right )
\leqslant \bar{g}_{k \left ( \bar{g}, z_i \right )} \left ( x_i \right )$
implies that
$\left | g_{k \left ( g, z_i \right )} \left ( x_i \right ) -
\bar{g}_{k \left ( \bar{g}, z_i \right )} \left ( x_i \right )
\right | \leqslant \theta_{k \left ( \bar{g}, z_i \right ) i}
m^{\frac{1}{p}} \epsilon$.
To sum up,
\begin{equation}
\label{eq:majorant-delta-g_k}
\forall i \in \ieg 1, m \ied, \;\;
\left | g_{k \left ( g, z_i \right )} \left ( x_i \right ) -
\bar{g}_{k \left ( \bar{g}, z_i \right )} \left ( x_i \right )
\right | \leqslant \max \left ( \theta_{k \left ( g, z_i \right ) i},
\theta_{k \left ( \bar{g}, z_i \right ) i}
\right ) m^{\frac{1}{p}} \epsilon.
\end{equation}
For all $k \in \ieg 1, C \ied$,
let $\boldsymbol{\theta}_k = \left ( \theta_{ki} 
\right )_{1 \leqslant i \leqslant m}$.
Making use once more of (\ref{eq:definition-of-bar-g}) provides us with:
\begin{equation}
\label{eq:normed-theta-k}
\forall k \in \ieg 1, C \ied, \;\; 
\left \| \boldsymbol{\theta}_k \right \|_p < 1.
\end{equation}
As a consequence,
\begin{align}
d_{p, \mathbf{z}_m} \left ( f_g, f_{\bar{g}} \right )
& = \;
\left ( \frac{1}{m} \sum_{i=1}^m \left | 
f_g \left ( z_i \right ) - f_{\bar{g}} \left ( z_i \right )
\right |^p \right )^{\frac{1}{p}} \nonumber \\
& = \; \frac{1}{2}
\left ( \frac{1}{m} \sum_{i=1}^m \left | g_{y_i} \left ( x_i \right )
- g_{k \left ( g, z_i \right )} \left ( x_i \right ) 
-  \bar{g}_{y_i} \left ( x_i \right )
+ \bar{g}_{k \left ( \bar{g}, z_i \right )} \left ( x_i \right )
\right |^p \right )^{\frac{1}{p}} \nonumber \\
& \leqslant \; \frac{1}{2}
\left ( \frac{1}{m} \sum_{i=1}^m \left ( \left | g_{y_i} \left ( x_i \right )
-  \bar{g}_{y_i} \left ( x_i \right ) \right | + \left |
g_{k \left ( g, z_i \right )} \left ( x_i \right ) 
- \bar{g}_{k \left ( \bar{g}, z_i \right )} \left ( x_i \right ) \right |
\right )^p \right )^{\frac{1}{p}} \nonumber \\
\label{eq:from-multivariate-to-univariate-partial-1}
& \leqslant \; \frac{1}{2} \left ( \sum_{i=1}^m
\left ( \theta_{y_i i} + \max \left ( \theta_{k \left ( g, z_i \right ) i},
\theta_{k \left ( \bar{g}, z_i \right ) i}
\right ) \right )^p
\right )^{\frac{1}{p}} \epsilon \\
& \leqslant \; \left ( \sum_{i=1}^m \max_{1 \leqslant k \leqslant C}
\theta_{ki}^p \right )^{\frac{1}{p}} \epsilon \nonumber \\
& \leqslant \; \left ( \sum_{k=1}^C 
\left \| \boldsymbol{\theta}_k \right \|_p^p 
\right )^{\frac{1}{p}} \epsilon \nonumber \\
\label{eq:from-multivariate-to-univariate-partial-2}
& < \; C^{\frac{1}{p}} \epsilon.
\end{align}
Inequality~(\ref{eq:from-multivariate-to-univariate-partial-1})
is obtained by application of (\ref{eq:def_theta_ik}) and
(\ref{eq:majorant-delta-g_k}),
and Inequality~(\ref{eq:from-multivariate-to-univariate-partial-2})
springs from
Inequality~(\ref{eq:normed-theta-k}).
We have established that the set of functions $f_{\bar{g}}$
is a proper $\left ( C^{\frac{1}{p}} \epsilon \right )$-net 
of $\mathcal{F}_{\mathcal{G}}$
with respect to the pseudo-metric
$d_{p, \mathbf{z}_m}$.
Since its cardinality is at most that of $\bar{\mathcal{G}}$,
$$
\forall \mathbf{z}_m \in \left ( \mathcal{X} \times \ieg 1, C \ied \right )^m,
\;\; \mathcal{N}^{(p)} \left ( C^{\frac{1}{p}} \epsilon, 
\mathcal{F}_{\mathcal{G}}, d_{p, \mathbf{z}_m} \right )
\leqslant
\prod_{k=1}^C \mathcal{N}^{(p)} \left ( \epsilon, \mathcal{G}_k,
d_{p, \mathbf{x}_m} \right ).
$$
The right-hand side inequality in
Formula~(\ref{eq:from-multivariate-to-univariate-L_p})
then follows from performing a change of variable.
The proof for the uniform convergence norm
results from taking the limit when $p$ goes to infinity.
\end{proof}
The actual generalized Sauer-Shelah lemma is an extension
of Lemma~3.5 in \cite{AloBenCesHau97} and Lemma~8 in \cite{BarLon95}.
In the case when $p$ is finite, then the upper bound is {\em dimension free}
(does not depend on the number $n$ of points) thanks to the implementation of 
the probabilistic extraction principle described in \cite{MenVer03}.

\begin{lemma}[Generalized Sauer-Shelah lemma]
\label{lemma:Lemma-3.5-in-AloBenCesHau97-in-L_p}
Let $\mathcal{F}$ be a class of functions from $\mathcal{T}$ into
$\left [ -M_{\mathcal{F}}, M_{\mathcal{F}} \right ]$
with $M_{\mathcal{F}} \in \mathbb{R}_+^*$.
$\mathcal{F}$ is supposed to be a GC class.
For $\epsilon \in \left ( 0, M_{\mathcal{F}} \right ]$,
let $d \left ( \epsilon \right ) 
= \epsilon\mbox{-dim} \left ( \mathcal{F} \right )$.
Then for
$\epsilon \in \left ( 0, 2 M_{\mathcal{F}} \right ]$
and $n \in \mathbb{N}^*$,
\begin{equation}
\label{eq:Lemma-3.5-in-AloBenCesHau97-in-L_p}
\forall p \in \mathbb{N}^*, \;\;
\mathcal{M}_p \left ( \epsilon, \mathcal{F}, n \right )
\leqslant
2^{2 \left ( K_{\epsilon} \left ( p \right ) + 1 \right )}
\left ( \frac{6272 e K_{\epsilon} \left ( p \right )}{3}
\left ( \frac{2 M_{\mathcal{F}}}{\epsilon} \right )^{2p+1}
\right )^{2 K_{\epsilon} \left ( p \right ) 
d \left ( \frac{\epsilon}{45} \right )},
\end{equation}
where $K_{\epsilon} \left ( p \right )
= \left \lceil \left ( p+2 \right ) 
\log_2 \left ( \left \lceil \frac{112 M_{\mathcal{F}}}{\epsilon} \right \rceil
\right ) \right \rceil$,
and
\begin{equation}
\label{eq:Lemma-3.5-in-AloBenCesHau97}
\mathcal{M}_{\infty} \left ( \epsilon, \mathcal{F}, n \right )
\leqslant 2 \left ( \frac{16 M_{\mathcal{F}}^2 n}{\epsilon^2}
\right )^{d \left ( \frac{\epsilon}{4} \right ) \log_2 \left (
\frac{4 M_{\mathcal{F}} e n}
{d \left ( \frac{\epsilon}{4} \right ) \epsilon} \right )}.
\end{equation}
\end{lemma}

\begin{proof}
Since~(\ref{eq:Lemma-3.5-in-AloBenCesHau97}) is simply an instance
of Lemma~3.5 in \cite{AloBenCesHau97}, we only prove
(\ref{eq:Lemma-3.5-in-AloBenCesHau97-in-L_p}).
By definition,
$$
\forall 
\mathbf{t}_n = \left( t_i \right)_{1 \leqslant i \leqslant n}
\in \mathcal{T}^n, \;\;
\mathcal{M} \left ( \epsilon, \mathcal{F}, d_{p, \mathbf{t}_n} \right )
= \mathcal{M} \left ( \epsilon, \left . \mathcal{F} \right |_{\mathbf{t}_n},
d_{p, \mathbf{t}_n} \right ),
$$
where $\left . \mathcal{F} \right |_{\mathbf{t}_n}$ is the set of
the restrictions to $\mathbf{t}_n$ of the functions in $\mathcal{F}$.
Let $\mathcal{F}_{\epsilon}$
be, among the subsets of $\left . \mathcal{F} \right |_{\mathbf{t}_n}$
$\epsilon$-separated with respect to the pseudo-metric $d_{p, \mathbf{t}_n}$,
a set of maximal cardinality. By definition,
$$
\left | \mathcal{F}_{\epsilon} \right | =
\mathcal{M} \left ( \epsilon, 
\left . \mathcal{F} \right |_{\mathbf{t}_n}, d_{p, \mathbf{t}_n} \right )
= \mathcal{M} \left ( \epsilon, \mathcal{F}_{\epsilon}, 
d_{p, \mathbf{t}_n} \right ).
$$
At this level, two cases must be considered.

\paragraph{First case}
Suppose that
$\left | \mathcal{F}_{\epsilon} \right | \leqslant
\exp \left ( K_e \left ( p \right ) n \epsilon^{2p} \right )$
where $K_e$ is the function of $p$ defined in 
Lemma~\ref{lemma:lemma-13-in-MenVer03-in-L_p}.
In that case, Lemma~\ref{lemma:lemma-13-in-MenVer03-in-L_p} applies,
and we can set $r$ equal to the smallest admissible value,
$\frac{\ln \left ( \left | \mathcal{F}_{\epsilon} \right | 
\right )} {K_e \left ( p \right ) \epsilon^{2p}}$, 
where $\ln$ is the Neperian (or natural) logarithm.
Consequently, there exists a subvector
$\mathbf{t}_q$ of $\mathbf{t}_n$ of size
\begin{equation}
\label{eq:partial-1-Lemma-3.5-in-AloBenCesHau97-in-L_p}
q \leqslant
\frac{\ln \left ( \left | \mathcal{F}_{\epsilon} \right | \right )} 
{K_e \left ( p \right ) \epsilon^{2p}}
\end{equation}
such that $\mathcal{F}_{\epsilon}$ is
$\left ( \left ( \frac{1}{2} \right )^{\frac{p+1}{p}} \epsilon 
\right )$-separated
with respect to the pseudo-metric $d_{p, \mathbf{t}_q}$,
and thus, since $\min_{p \in \mathbb{N}^*}
\left ( \frac{1}{2} \right )^{\frac{p+1}{p}} = \frac{1}{4}$,
$\frac{\epsilon}{4}$-separated with respect to the same pseudo-metric.
As a consequence,
$$
\left | \mathcal{F}_{\epsilon} \right |
= \mathcal{M} \left ( \frac{\epsilon}{4}, \mathcal{F}_{\epsilon}, 
d_{p, \mathbf{t}_q} \right )
= \mathcal{M} \left ( \frac{\epsilon}{4}, 
\left . \mathcal{F}_{\epsilon} \right |_{\mathbf{t}_q}, 
d_{p, \mathbf{t}_q} \right )
= \left | \left . \mathcal{F}_{\epsilon} \right |_{\mathbf{t}_q} \right |.
$$
For $\eta \in \left ( 0, \frac{\epsilon}{4} \right )$,
let $\left ( \left . \mathcal{F}_{\epsilon} \right |_{\mathbf{t}_q} 
\right )^{\left ( \eta \right )}$
be the image of $\left . \mathcal{F}_{\epsilon} \right |_{\mathbf{t}_q}$ 
by the discretization operator
$\left ( \cdot \right )^{\eta}$.
Since $\left | \left . \mathcal{F}_{\epsilon} \right |_{\mathbf{t}_q} \right |
= \mathcal{M} \left ( \frac{\epsilon}{4}, 
\left . \mathcal{F}_{\epsilon} \right |_{\mathbf{t}_q}, 
d_{p, \mathbf{t}_q} \right )$,
by application of Lemma~\ref{lemma:lemma-3.2.2-in-AloBenCesHau97-in-L_p},
\begin{align*}
\mathcal{M} \left ( \frac{\epsilon}{4}, 
\left . \mathcal{F}_{\epsilon} \right |_{\mathbf{t}_q}, 
d_{p, \mathbf{t}_q} \right )
& = \; \mathcal{M} 
\left ( \frac{\left ( \left ( \frac{\epsilon}{4} \right )^p
- \left ( \eta \right )^p \right )^{\frac{1}{p}}}{2},
\left ( \left . \mathcal{F}_{\epsilon} \right |_{\mathbf{t}_q} 
\right )^{\left ( \eta \right )}, d_{p, \mathbf{t}_q} \right ) \\
& \leqslant \; \mathcal{M} 
\left ( \frac{\epsilon - 4 \eta}{8},
\left ( \left . \mathcal{F}_{\epsilon} \right |_{\mathbf{t}_q} 
\right )^{\left ( \eta \right )}, d_{p, \mathbf{t}_q} \right ),
\end{align*}
where the inequality stems from the fact that
$$
\min_{p \in \mathbb{N}^*} \left ( \left ( \frac{\epsilon}{4} \right )^p
- \left ( \eta \right )^p \right )^{\frac{1}{p}} 
= \frac{\epsilon}{4} - \eta.
$$
For $N \in \mathbb{N}$ satisfying 
$N > \frac{56 M_{\mathcal{F}}}{\epsilon}$, 
let us set $\eta = \frac{2 M_{\mathcal{F}}}{N}$.
Since $\left ( \left . \mathcal{F}_{\epsilon} \right |_{\mathbf{t}_q} 
\right )^{\left ( \frac{2 M_{\mathcal{F}}}{N} \right )}$ 
is a class of functions whose domain has
cardinality $q$ and whose codomain is $\left \{ 2 M_{\mathcal{F}} \frac{j}{N}: 
0 \leqslant j \leqslant N \right \}$, 
Lemma~\ref{lemma:Lemma-8-in-BarLon95-in-L_p} provides us with
$$
\left | \mathcal{F}_{\epsilon} \right | \leqslant
2^{\left ( p+2 \right ) \log_2 \left ( N \right ) + 1} 
\left ( \frac{e \left ( N-1 \right ) q}{d_1} 
\right )^{\left ( p+2 \right ) \log_2 \left ( N \right ) d_1}
$$
where
$d_1 = \left ( \frac{1}{16} \left ( \epsilon - \frac{56 M_{\mathcal{F}}}{N}
\right ) \right )\mbox{-dim} \left ( \left ( \left . \mathcal{F}_{\epsilon} 
\right |_{\mathbf{t}_q} 
\right )^{\left ( \frac{2 M_{\mathcal{F}}}{N} \right )} \right )$.
Thus, making use of the upper bound on $q$ provided
by (\ref{eq:partial-1-Lemma-3.5-in-AloBenCesHau97-in-L_p}),
\begin{align}
\left | \mathcal{F}_{\epsilon} \right | 
& \leqslant \;
2^{\left ( p+2 \right ) \log_2 \left ( N \right ) + 1} 
\left ( \frac{e \left ( N-1 \right ) 
\ln \left ( \left | \mathcal{F}_{\epsilon} \right | \right )}
{K_e \left ( p \right ) \epsilon^{2p} d_1} 
\right )^{\left ( p+2 \right ) \log_2 \left ( N \right ) d_1} \nonumber \\
\label{eq:partial-5-Lemma-3.5-in-AloBenCesHau97-in-L_p}
& \leqslant \;
2^{\left ( p+2 \right ) \log_2 \left ( N \right ) + 1}
\left (
\frac{\ln \left ( \left | \mathcal{F}_{\epsilon} \right | \right )}
{d_1} \right )^{\left ( p+2 \right ) \log_2 \left ( N \right ) d_1}
\left ( \frac{e \left ( N-1 \right )}
{K_e \left ( p \right ) \epsilon^{2p}} \right )^{
\left ( p+2 \right ) \log_2 \left ( N \right ) d_1}.
\end{align}
For all $r \in \mathbb{N}^*$, let $h_r$
be the function on $\left [ 1, +\infty \right )$
mapping $u$ to 
$2^{r-1} r!~u^{\frac{1}{2}} - \ln^r \left ( u \right )$.
The function $h_1$ is positive on its domain and for all $r \in \mathbb{N}^*$,
$h_r \left ( 1 \right ) > 0$.
Since for all $r \geqslant 2$, $h_r' \left ( u \right ) = 
\frac{r}{u} h_{r-1} \left ( u \right )$,
proceeding by induction, one establishes that
all the functions $h_r$ are positive on their domain.
Furthermore, for all $r \in \mathbb{N}^*$, 
$2^{r-1} r! \leqslant r^r$.
Consequently, setting 
$K_{N,p} = \left \lceil \left ( p+2 \right ) 
\log_2 \left ( N \right ) \right \rceil$,
where the ceiling function $\left \lceil \cdot \right \rceil$ is defined by:
$$
\forall u \in \mathbb{R}, \;\; \left \lceil u \right \rceil =
\min \left \{ j \in \mathbb{Z}: \; j \geqslant u \right \},
$$
we obtain
\begin{align}
\left (
\frac{\ln \left ( \left | \mathcal{F}_{\epsilon} \right | \right )}
{d_1} \right )^{\left ( p+2 \right ) \log_2 \left ( N \right ) d_1}
& = \left ( \ln^{\left ( p+2 \right ) \log_2 \left ( N \right )}
\left ( \left | \mathcal{F}_{\epsilon} \right |^{\frac{1}{d_1}} 
\right ) \right )^{d_1} \nonumber \\
& \leqslant \;
\left ( \ln^{K_{N,p}}
\left ( \left | \mathcal{F}_{\epsilon} 
\right |^{\frac{1}{d_1}} \right ) \right )^{d_1} \nonumber \\
\label{eq:partial-6-Lemma-3.5-in-AloBenCesHau97-in-L_p}
& < \;
K_{N,p}^{K_{N,p} d_1}
\left | \mathcal{F}_{\epsilon} \right |^{\frac{1}{2}}.
\end{align}
A substitution of the right-hand side of
(\ref{eq:partial-6-Lemma-3.5-in-AloBenCesHau97-in-L_p})
into (\ref{eq:partial-5-Lemma-3.5-in-AloBenCesHau97-in-L_p}) gives
$$
\left | \mathcal{F}_{\epsilon} \right | 
\leqslant
2^{2 \left ( K_{N,p} + 1 \right )}
\left ( \frac{e \left ( N-1 \right ) K_{N,p}}
{K_e \left ( p \right ) \epsilon^{2p}} \right )^{2 K_{N,p} d_1}.
$$
To bound from above $d_1$, $N$ can be set equal to
$\left \lceil \frac{112 M_{\mathcal{F}}}{\epsilon} \right \rceil$.
Then,
\begin{align}
\label{eq:partial-3-Lemma-3.5-in-AloBenCesHau97-in-L_p}
d_1 & \leqslant \;
\left ( \frac{\epsilon}{32} 
\right )\mbox{-dim} \left ( \left ( \left . \mathcal{F}_{\epsilon} 
\right |_{\mathbf{t}_q} 
\right )^{\left ( \frac{2 M_{\mathcal{F}}}{
\left \lceil \frac{112 M_{\mathcal{F}}}{\epsilon} \right \rceil} \right )} 
\right ) \\
\label{eq:partial-4-Lemma-3.5-in-AloBenCesHau97-in-L_p}
& \leqslant \;
\left ( \frac{\epsilon}{32} - \frac{M_{\mathcal{F}}}{
\left \lceil \frac{112 M_{\mathcal{F}}}{\epsilon} \right \rceil}
\right )\mbox{-dim} \left ( \left . \mathcal{F}_{\epsilon} 
\right |_{\mathbf{t}_q} \right ) \\
& \leqslant \;
\left ( \frac{\epsilon}{32} - \frac{\epsilon}{112}
\right )\mbox{-dim} \left ( \left . \mathcal{F}_{\epsilon} 
\right |_{\mathbf{t}_q} \right ) \nonumber \\
& \leqslant \;
\left ( \frac{\epsilon}{45} \right )\mbox{-dim} \left ( \mathcal{F} \right ).
\nonumber
\end{align}
This sequence of computations makes use three times of
the fact that the fat-shattering dimension is a nonincreasing function
of the margin parameter. The transition from 
(\ref{eq:partial-3-Lemma-3.5-in-AloBenCesHau97-in-L_p}) to
(\ref{eq:partial-4-Lemma-3.5-in-AloBenCesHau97-in-L_p})
is provided by Lemma~\ref{lemma:lemma-3.2.1-in-AloBenCesHau97-in-L_p}.
As a consequence,
$$
\left | \mathcal{F}_{\epsilon} \right | 
\leqslant
2^{2 \left ( \left \lceil \left ( p+2 \right ) 
\log_2 \left ( \left \lceil \frac{112 M_{\mathcal{F}}}{\epsilon} \right \rceil
\right ) \right \rceil + 1 \right )}
\left ( \frac{112 e M_{\mathcal{F}}
\left \lceil \left ( p+2 \right ) 
\log_2 \left ( \left \lceil \frac{112 M_{\mathcal{F}}}{\epsilon} \right \rceil
\right ) \right \rceil}
{K_e \left ( p \right ) \epsilon^{2p+1}} 
\right )^{2 \left \lceil \left ( p+2 \right ) 
\log_2 \left ( \left \lceil \frac{112 M_{\mathcal{F}}}{\epsilon} \right \rceil
\right ) \right \rceil d \left ( \frac{\epsilon}{45} \right )}.
$$
A substitution into the right-hand side of the value of
$K_e \left ( p \right )$ produces for 
$\left | \mathcal{F}_{\epsilon} \right |$, i.e.,
$\mathcal{M} \left ( \epsilon, \mathcal{F}, d_{p, \mathbf{t}_n} \right )$,
the same upper bound
as that announced for $\mathcal{M}_p \left ( \epsilon, \mathcal{F}, n \right )$.
Thus, to conclude the proof of
(\ref{eq:Lemma-3.5-in-AloBenCesHau97-in-L_p})
under the assumption that
$\left | \mathcal{F}_{\epsilon} \right | \leqslant
\exp \left ( K_e \left ( p \right ) n \epsilon^{2p} \right )$,
it suffices to notice that this upper bound does not depend on $\mathbf{t}_n$
(it is even dimension free).

\paragraph{Second case}
Suppose conversely that
$\left | \mathcal{F}_{\epsilon} \right | >
\exp \left ( K_e \left ( p \right ) n \epsilon^{2p} \right )$,
i.e.,
\begin{equation}
\label{eq:partial-2-Lemma-3.5-in-AloBenCesHau97-in-L_p}
n < \frac{\ln \left ( \left | \mathcal{F}_{\epsilon} \right | \right )} 
{K_e \left ( p \right ) \epsilon^{2p}}.
\end{equation}
By application of Lemma~\ref{lemma:lemma-3.2.2-in-AloBenCesHau97-in-L_p},
for $\eta \in \left ( 0, \epsilon \right )$,
\begin{align*}
\left | \mathcal{F}_{\epsilon} \right | 
& = \;
\mathcal{M} 
\left ( \frac{\left ( \epsilon^p - \eta^p \right )^{\frac{1}{p}}}{2},
\left ( \mathcal{F}_{\epsilon} \right )^{(\eta)}, 
d_{p, \mathbf{t}_n} \right ) \\
& \leqslant \;
\mathcal{M} 
\left ( \frac{\epsilon - \eta}{2},
\left ( \mathcal{F}_{\epsilon} \right )^{(\eta)}, 
d_{p, \mathbf{t}_n} \right ).
\end{align*}
For $N \in \mathbb{N}$ satisfying
$N > \frac{14 M_{\mathcal{F}}}{\epsilon}$,
let us set $\eta = \frac{2 M_{\mathcal{F}}}{N}$.
Since
$\left ( \mathcal{F}_{\epsilon} 
\right )^{\left ( \frac{2 M_{\mathcal{F}}}{N} \right )}$
is a class of functions whose domain has
cardinality $n$ and whose codomain is $\left \{ 2 M_{\mathcal{F}} \frac{j}{N}: 
0 \leqslant j \leqslant N \right \}$,
Lemma~\ref{lemma:Lemma-8-in-BarLon95-in-L_p} provides us with
$$
\left | \mathcal{F}_{\epsilon} \right | \leqslant
2^{\left ( p+2 \right ) \log_2 \left ( N \right ) + 1} 
\left ( \frac{e \left ( N-1 \right ) n}{d_2} 
\right )^{\left ( p+2 \right ) \log_2 \left ( N \right ) d_2}
$$
where $d_2 = \left ( \frac{1}{4} \left ( \epsilon 
- \frac{14 M_{\mathcal{F}}}{N} \right ) \right )\mbox{-dim} 
\left ( \left ( \mathcal{F}_{\epsilon} 
\right )^{(\frac{2 M_{\mathcal{F}}}{N})} \right )$.
The substitution of the upper bound on $n$ provided by
(\ref{eq:partial-2-Lemma-3.5-in-AloBenCesHau97-in-L_p})
into this bound produces
\begin{equation}
\label{eq:partial-7-Lemma-3.5-in-AloBenCesHau97-in-L_p}
\left | \mathcal{F}_{\epsilon} \right | 
\leqslant
2^{2 \left ( K_{N,p} + 1 \right )}
\left ( \frac{e \left ( N-1 \right ) K_{N,p}}
{K_e \left ( p \right ) \epsilon^{2p}} \right )^{2 K_{N,p} d_2}.
\end{equation}
To bound from above $d_2$, $N$ can be set equal to
$\left \lceil \frac{28 M_{\mathcal{F}}}{\epsilon} \right \rceil$.
Then, the line of reasoning used for $d_1$ leads to
$$
d_2 \leqslant
\left ( \frac{\epsilon}{12} \right )\mbox{-dim} \left ( \mathcal{F} \right ).
$$
By substitution into (\ref{eq:partial-7-Lemma-3.5-in-AloBenCesHau97-in-L_p})
of the value of $N$ and this upper bound on $d_2$, an upper bound on
$\left | \mathcal{F}_{\epsilon} \right |$ is obtained which is smaller than
that provided by Inequality~(\ref{eq:Lemma-3.5-in-AloBenCesHau97-in-L_p}).
\end{proof}

\subsection{Comparison with the state of the art}

In order to limit the complexity of the formula corresponding
to finite values of $p$
(Inequality~(\ref{eq:Lemma-3.5-in-AloBenCesHau97-in-L_p})),
the constants have systematically been derived by considering 
the ``worst'' case: $p=1$. This implies that better constants
can be obtained by focusing on the value of $p$ of interest.
If the resulting gain is all the more important as this value is large,
it is already noticeable for $p=2$.
The result that compares directly with 
Lemma~\ref{lemma:Lemma-3.5-in-AloBenCesHau97-in-L_p} is
Theorem~3.2 in \cite{Men02}.
As Inequality~(\ref{eq:Lemma-3.5-in-AloBenCesHau97-in-L_p}),
the corresponding bound is dimension free.
The main difference rests in the dependency on
the fat-shattering dimension.
Whereas Inequality~(\ref{eq:Lemma-3.5-in-AloBenCesHau97-in-L_p})
corresponds to a growth rate of the $\epsilon$-entropy
with this dimension which is linear, 
Theorem~3.2 in \cite{Men02} exhibits an additional
logarithmic multiplicative factor.
Focusing on results derived for a specific $L_p$-norm,
the literature provides us with one example of generalized
Sauer-Shelah lemma based on the $L_1$-norm: Lemma~1 in \cite{BarKulPos97}
(whose basic combinatorial result is Lemma~8 in \cite{BarLon95}).
However, this result is not dimension free
(the growth rate of the $\epsilon$-entropy with $n$ is logarithmic).
As for the $L_2$-norm, the state of the art is provided by
Theorem~1 in \cite{MenVer03}. Since its original formulation 
involves unspecified universal constants,
to make comparison possible,
it is given below with explicit constants.

\begin{lemma}[After Theorem~1 in \cite{MenVer03}]
\label{lemma:generalized-Sauer-Shelah-lemma-in-L_2-MenVer03}
Let $\mathcal{F}$ be a class of functions from $\mathcal{T}$ into
$\left [ -M_{\mathcal{F}}, M_{\mathcal{F}} \right ]$
with $M_{\mathcal{F}} \in \mathbb{R}_+^*$.
$\mathcal{F}$ is supposed to be a GC class.
For $\epsilon \in \left ( 0, M_{\mathcal{F}} \right ]$,
let $d \left ( \epsilon \right )
= \epsilon\mbox{-dim} \left ( \mathcal{F} \right )$.
Then for
$\epsilon \in \left ( 0, 2 M_{\mathcal{F}} \right ]$
and $n \in \mathbb{N}^*$,
\begin{equation}
\label{eq:generalized-Sauer-Shelah-lemma-in-L_2-MenVer03}
\mathcal{M}_2 \left ( \epsilon, \mathcal{F}, n \right )
\leqslant \left ( 3584 e
\left ( \frac{2 M_{\mathcal{F}}}{\epsilon} \right )^5
\right )^{4 d \left ( \frac{\epsilon}{96} \right )}.
\end{equation}
\end{lemma}
With this formulation at hand, it appears that 
even without optimizing the constants of 
Inequality~(\ref{eq:Lemma-3.5-in-AloBenCesHau97-in-L_p})
for the case $p=2$, none of the two bounds
is uniformly better than the other. The choice between them
should primarily be based on the behaviour of the fat-shattering dimensions
of interest.

\section{Bound based on the $L_{\infty}$-norm}
\label{sec:L_infty-bound}

The $L_{\infty}$-norm plays a central part in the theory of bounds.
Indeed, one can consider that it is already at the core of the initial result
of Vapnik and Chervonenkis \cite{VapChe71}.
Focusing on margin classifiers, it is the norm used in Bartlett's
seminal article \cite{Bar98}.

\subsection{State of the art}

To the best of our knowledge, the state-of-the-art result is precisely
a multi-class extension of Bartlett's result: Theorem~40 in \cite{Gue07b}.
It makes use of the same margin loss functions, defined as follows.

\begin{definition}[Margin loss functions $\phi_{\infty, \gamma}$]
\label{def:margin-loss-function-L_infty}
For $\gamma \in \left (0, 1 \right ]$,
the margin loss function $\phi_{\infty, \gamma}$ is defined by:
$$
\forall t \in \mathbb{R}, \;\;
\phi_{\infty, \gamma} \left ( t \right ) = 
\nbOne_{\left \{ t < \gamma \right \}}.
$$
\end{definition}
The basic supremum inequality is
a multi-class extension of Lemma~4 in \cite{Bar98},
with the first symmetrization being derived from the basic lemma
of Section~4.5.1 in \cite{Vap98}.

\begin{theorem}[After Theorem~22 in \cite{Gue07b}]
\label{theorem:basic-supremum-bound-uniform-convergence-norm}
Let $\mathcal{G}$ be a class of functions
satisfying Definition~\ref{def:margin-multi-category-classifiers}.
For $\gamma \in \left ( 0, 1 \right ]$,
let $\mathcal{F}_{\mathcal{G}, \gamma}$ be the class of functions deduced
from $\mathcal{G}$ according to
Definition~\ref{def:class-of-regret-functions}.
For a fixed $\gamma \in \left ( 0, 1 \right ]$
and a fixed $\delta \in \left ( 0, 1 \right )$,
with $P^m$-probability at least $1 - \delta$, uniformly for every function
$g \in \mathcal{G}$,
\begin{equation}
\label{eq:basic-supremum-bound-uniform-convergence-norm}
L \left ( g \right ) \leqslant L_{\gamma, m} \left ( g \right )
+ \sqrt{ \frac{2}{m} \left ( \ln \left ( 
\mathcal{N}_{\infty}^{(p)} \left ( \frac{\gamma}{2},
\mathcal{F}_{\mathcal{G}, \gamma},  2m \right ) \right )
+ \ln \left ( \frac{2}{\delta} \right )
\right )} + \frac{1}{m},
\end{equation}
where the margin loss function defining the empirical margin risk
is $\phi_{\infty, \gamma}$
(Definition~\ref{def:margin-loss-function-L_infty}).
\end{theorem}
The pathway leading from this inequality to Theorem~40 in \cite{Gue07b}
consists in relating the covering number of interest to
a $\gamma$-$\Psi$-dimension (see Definition~28 in \cite{Gue07b})
of a class of vector-valued functions.
The dependency on $C$ varies with the choice of this dimension.
In the case of the dimension which is the easiest to bound from above
(by application of the pigeonhole principle), the margin Natarajan dimension,
it is superlinear.

\subsection{Improved dependency on $C$}
Instead of working with vector-valued functions as in \cite{Gue07b}, 
it is more efficient
to handle separately the classes of component functions.
Starting from 
Inequality~(\ref{eq:basic-supremum-bound-uniform-convergence-norm})
and applying in sequence Lemma~\ref{lemma:from-multivariate-to-univariate-L_p}
(for $p = \infty$),
Lemma~\ref{lemma:Kolmogorov}
and Lemma~\ref{lemma:Lemma-3.5-in-AloBenCesHau97-in-L_p} 
(Lemma~3.5 in \cite{AloBenCesHau97}) produces
the master theorem in the uniform convergence norm.

\begin{theorem}
\label{theorem:final-guaranteed-risk-L_infty}
Let $\mathcal{G}$ be a class of functions
satisfying Definition~\ref{def:margin-multi-category-classifiers}.
For $\epsilon \in \left ( 0, M_{\mathcal{G}} \right ]$,
let $d \left ( \epsilon \right ) = \max_{1 \leqslant k \leqslant C}
\epsilon\mbox{-dim} \left ( \mathcal{G}_k \right )$.
For a fixed $\gamma \in \left ( 0, 1 \right ]$
and a fixed $\delta \in \left ( 0, 1 \right )$,
with $P^m$-probability at least $1 - \delta$, uniformly for every function
$g \in \mathcal{G}$,
$$
L \left ( g \right ) \leqslant
L_{\gamma, m} \left ( g \right )
+ \sqrt{ \frac{2}{m} \left ( 3 C
d \left ( \frac{\gamma}{8} \right )
\ln^2 \left ( \frac{128 M_{\mathcal{G}}^2 m}{\gamma^2} 
\right )
+ \ln \left ( \frac{2}{\delta} \right )
\right )} + \frac{1}{m}.
$$
\end{theorem}

\begin{proof}
The sketch of the proof has been given at the beginning of the subsection.
The detail makes use of the fact that
$$
\forall k \in \ieg 1, C \ied, \;\; \log_2 \left ( 
\frac{16 M_{\mathcal{G}} e m}
{\left ( \frac{\gamma}{8} \right )\mbox{-dim} \left ( \mathcal{G}_k \right )
\gamma} \right ) \leqslant \frac{1}{\ln \left ( 2 \right )}
\ln \left ( \frac{128 M_{\mathcal{G}}^2 m}{\gamma^2} \right ).
$$
Thus,
\begin{align*}
L \left ( g \right ) & \leqslant \;
L_{\gamma, m} \left ( g \right )
+ \sqrt{ \frac{2}{m} \left ( \sum_{k=1}^C \ln \left ( 
\mathcal{N}_{\infty}^{(p)} \left ( \frac{\gamma}{2},
\mathcal{G}_k,  2m \right ) \right )
+ \ln \left ( \frac{2}{\delta} \right )
\right )} + \frac{1}{m} \\
& \leqslant \;
L_{\gamma, m} \left ( g \right )
+ \sqrt{ \frac{2}{m} \left ( \frac{2}{\ln \left ( 2 \right )}
\ln \left ( \frac{128 M_{\mathcal{G}}^2 m}{\gamma^2} \right )
\sum_{k=1}^C 
\left ( \frac{\gamma}{8} \right )\mbox{-dim} \left ( \mathcal{G}_k \right )
\ln \left ( 
\frac{16 M_{\mathcal{G}} e m}
{\left ( \frac{\gamma}{8} \right )\mbox{-dim} \left ( \mathcal{G}_k \right )
\gamma} \right )
+ \ln \left ( \frac{2}{\delta} \right )
\right )} + \frac{1}{m} \\
& \leqslant \;
L_{\gamma, m} \left ( g \right )
+ \sqrt{ \frac{2}{m} \left ( 3
\ln^2 \left ( \frac{128 M_{\mathcal{G}}^2 m}{\gamma^2} \right )
\sum_{k=1}^C 
\left ( \frac{\gamma}{8} \right )\mbox{-dim} \left ( \mathcal{G}_k \right )
+ \ln \left ( \frac{2}{\delta} \right )
\right )} + \frac{1}{m} \\
& \leqslant \; 
L_{\gamma, m} \left ( g \right )
+ \sqrt{ \frac{2}{m} \left ( 3 C
d \left ( \frac{\gamma}{8} \right )
\ln^2 \left ( \frac{128 M_{\mathcal{G}}^2 m}{\gamma^2} 
\right )
+ \ln \left ( \frac{2}{\delta} \right )
\right )} + \frac{1}{m}.
\end{align*}
\end{proof}

\subsection{Discussion}
Under the assumption that $d \left ( \epsilon \right )$ does not depend on $C$,
Theorem~\ref{theorem:final-guaranteed-risk-L_infty}
provides a guaranteed risk whose control term varies with $C$ and $m$
as a $O\left ( \ln \left ( m \right ) \sqrt{\frac{C}{m}} \right )$.
To sum up, the new bound exhibits the convergence rate of
Theorem~40 in \cite{Gue07b}, whereas its control term grows
only as the square root of $C$. Note that Lemma~19 in \cite{Zha04},
which provides a bound with the same growth,
holds for kernel multi-category classification methods only.
We now establish an improvement of this kind with the $L_2$-norm.

\section{Bound based on the $L_2$-norm}
\label{sec:L_2-bound}

As in the case of the uniform convergence norm,
the state-of-the-art result provides us not only with an element of comparison,
but also with a starting point for the derivation of our guaranteed risk.

\subsection{State of the art}

The sharpest bound in the $L_2$-norm is Theorem~3 in \cite{KuzMohSye14}.
The margin loss function involved in this result is a standard one,
the parameterized truncated hinge loss
(that satisfies both Definition~\ref{def:margin-loss-functions}
and the definition used by Koltchinskii and Panchenko in \cite{KolPan02}).

\begin{definition}[Parameterized truncated hinge loss $\phi_{2, \gamma}$,
Definition~4.3 in \cite{MohRosTal12}]
\label{def:margin-loss-function-L_2}
For $\gamma \in \left (0, 1 \right ]$,
the {\em parameterized truncated hinge loss} $\phi_{2, \gamma}$ is defined by:
$$
\forall t \in \mathbb{R}, \;\;
\phi_{2, \gamma} \left ( t \right ) = \nbOne_{\left \{ t \leqslant 0 \right \}}
+ \left ( 1 - \frac{t}{\gamma} \right )
\nbOne_{\left \{ t \in \left ( 0, \gamma \right ] \right \}}.
$$
\end{definition}
This guaranteed risk is built upon a basic supremum inequality which is
a partial result in the proof of Theorem~8.1 in \cite{MohRosTal12}
(with $\mathcal{F}_{\mathcal{G}}$ replaced with
$\mathcal{F}_{\mathcal{G}, \gamma}$). 

\begin{theorem}[After Theorem~8.1 in \cite{MohRosTal12}]
\label{theorem:basic-supremum-bound-L_2-norm}
Let $\mathcal{G}$ be a class of functions
satisfying Definition~\ref{def:margin-multi-category-classifiers}.
For $\gamma \in \left ( 0, 1 \right ]$,
let $\mathcal{F}_{\mathcal{G}, \gamma}$ be the class of functions deduced
from $\mathcal{G}$ according to
Definition~\ref{def:class-of-regret-functions}.
For a fixed $\gamma \in \left ( 0, 1 \right ]$
and a fixed $\delta \in \left ( 0, 1 \right )$,
with $P^m$-probability at least $1 - \delta$, uniformly for every function
$g \in \mathcal{G}$,
$$
L \left ( g \right ) \leqslant
L_{\gamma, m} \left ( g \right )
+ \frac{2}{\gamma} 
R_m \left ( \mathcal{F}_{\mathcal{G}, \gamma} \right )
+ \sqrt{\frac{\ln \left ( \frac{1}{\delta} \right ) }{2m}}
$$
where the margin loss function defining the empirical margin risk
is the parameterized truncated hinge loss 
(Definition~\ref{def:margin-loss-function-L_2}).
\end{theorem}
Theorem~3 in \cite{KuzMohSye14} stems from
Theorem~\ref{theorem:basic-supremum-bound-L_2-norm}
by application of the following lemma.

\begin{lemma}
Let $\mathcal{G}$ be a class of functions
satisfying Definition~\ref{def:margin-multi-category-classifiers}.
For $\gamma \in \left ( 0, 1 \right ]$,
let $\mathcal{F}_{\mathcal{G}, \gamma}$ be the class of functions deduced
from $\mathcal{G}$ according to
Definition~\ref{def:class-of-regret-functions}. Then
\begin{equation}
\label{eq:K-M-S}
R_m \left ( \mathcal{F}_{\mathcal{G}, \gamma} \right )
\leqslant C R_m \left ( \bigcup_{k=1}^C \mathcal{G}_k \right ).
\end{equation}
\end{lemma}
Many margin classifiers, including neural networks and kernel machines,
satisfy the additional property that all the classes of component
functions are identical,
so that the growth with $C$ of the upper bound on
$R_m \left ( \mathcal{F}_{\mathcal{G}, \gamma} \right )$
provided by (\ref{eq:K-M-S}) is linear.
Furthermore, if the classifier is specifically a kernel machine,
then it is well known that by
combining the reproducing property with the Cauchy-Schwarz
inequality, it is possible to obtain an upper bound on the Rademacher
complexity which is a $O \left ( m^{-\frac{1}{2}} \right )$
(see for instance Lemma~22 in \cite{BarMen02}).
Thus, for kernel machines, the control term of Kuznetsov's bound
is a $O \left ( \frac{C}{\sqrt{m}} \right )$.
Kernel machines (with bounded range) satisfy
Definition~\ref{def:margin-multi-category-classifiers}.
This is easy to establish thanks to the characterization
of the GC classes provided by Theorem~2.5 in \cite{AloBenCesHau97}.
The finiteness of the $\gamma$-dimension
of a linear separator in a reproducing kernel Hilbert space
is a well-known result, which appears, for instance,
as a consequence of Theorem~4.6 in \cite{BarSha99}.
To sum up, the state-of-the-art result is a guaranteed risk
whose control term is at best a $O \left ( \frac{C}{\sqrt{m}} \right )$,
for a specific family of classifiers among those satisfying
Definition~\ref{def:margin-multi-category-classifiers}.

\subsection{Improved dependency on $C$}
\label{sec:chaining-and-decomposition}

Several results are available to bound from above
the expected suprema of empirical processes
(see for instance Chapters~1, 2, and 6 of \cite{Mas07}).
We resort to the standard approach, especially efficient
in the case of Rademacher processes,
the application of Dudley's chaining method \cite{Dud67}.

\begin{theorem}[Chained bound on the Rademacher complexity
of $\mathcal{F}_{\mathcal{G}, \gamma}$]
\label{theorem:final-guaranteed-risk-L_2}
Let $\mathcal{G}$ be a class of functions
satisfying Definition~\ref{def:margin-multi-category-classifiers}.
For $\gamma \in \left ( 0, 1 \right ]$,
let $\mathcal{F}_{\mathcal{G}, \gamma}$ be the class of functions deduced
from $\mathcal{G}$ according to
Definition~\ref{def:class-of-regret-functions}.
For $\epsilon \in \left ( 0, M_{\mathcal{G}} \right ]$, let
$d \left ( \epsilon \right ) = \max_{1 \leqslant k \leqslant C}
\epsilon\mbox{-dim} \left ( \mathcal{G}_k \right )$.
Let $h$ be a positive and decreasing function on $\mathbb{N}$ such that
$h \left ( 0 \right ) \geqslant \gamma$ and 
$h \left ( 1 \right ) \leqslant 2 M_{\mathcal{G}} \sqrt{C}$.
Then for all $N \in \mathbb{N}^*$,
\begin{equation}
\label{eq:chaining-sum-final}
R_m \left ( \mathcal{F}_{\mathcal{G}, \gamma} \right )
\leqslant
h \left ( N \right )
+ 4 \sqrt{\frac{5 C}{m}} 
\sum_{j=1}^N \left ( h \left ( j \right ) + h \left ( j-1 \right ) \right )
\sqrt{d \left ( \frac{h \left ( j \right )}{96 \sqrt{C}} \right )
\ln \left ( \frac{14 M_{\mathcal{G}} \sqrt{C}}{h \left ( j \right )} \right )}.
\end{equation}
\end{theorem}

\begin{proof}
The initial part of the proof of Formula~(\ref{eq:chaining-sum-final})
is the application of
Theorem~\ref{theorem:Dudley's-metric-entropy-bound}.
Note that 
$\text{diam} \left ( \mathcal{F}_{\mathcal{G}, \gamma} \right ) 
\leqslant \gamma$, justifying the hypothesis on $h \left ( 0 \right )$.
An advantage of working with $\mathcal{F}_{\mathcal{G}, \gamma}$ 
instead of $\mathcal{F}_{\mathcal{G}}$ (directly) has thus been highlighted.
The end of the proof consists in applying in sequence
Lemma~\ref{lemma:from-multivariate-to-univariate-L_p} (for $p=2$),
Lemma~\ref{lemma:Kolmogorov} and 
Lemma~\ref{lemma:generalized-Sauer-Shelah-lemma-in-L_2-MenVer03}
(with $\epsilon = \frac{h \left ( j \right )}{\sqrt{C}}$
and $3584 e$ bounded from above by $7^5$).
\end{proof}
Thanks to the choice $h \left ( j \right ) = 2^{-j} \sqrt{C} \gamma$, 
under the assumption that $d \left ( \epsilon \right )$
does not depend on $C$, then Theorem~\ref{theorem:final-guaranteed-risk-L_2}
provides a guaranteed risk whose control term grows linearly with $C$,
a dependency at least as good as that of Theorem~3 in \cite{KuzMohSye14}.
The improvement announced results from substituting to the hypothesis of
GC classes a slightly stronger one.

\begin{hypothesis}
\label{hypothesis:restriction-gamma-dim}
We consider classes of functions $\mathcal{G}$
satisfying Definition~\ref{def:margin-multi-category-classifiers}
plus the fact that there exists a pair
$\left ( d_{\mathcal{G}}, K_{\mathcal{G}} \right )
\in \mathbb{N}^* \times \mathbb{R}_+^*$ such that 
\begin{equation}
\label{eq:gamma-dim-bound}
\forall \epsilon \in \left ( 0, M_{\mathcal{G}} \right ], \;\; 
\max_{1 \leqslant k \leqslant C}
\epsilon\mbox{-dim} \left ( \mathcal{G}_k \right )
\leqslant K_{\mathcal{G}} \epsilon^{- d_{\mathcal{G}}}.
\end{equation}
\end{hypothesis}
If Hypothesis~\ref{hypothesis:restriction-gamma-dim} is satisfied,
then the classes $\mathcal{G}_k$ are 
{\em universal Donsker classes} \cite{Men02}.
Theorem~4.6 in \cite{BarSha99} tells us that it is the case,
with $d_{\mathcal{G}} = 2$, if each of the classes $\mathcal{G}_k$
corresponds to the class of functions computed by a support vector machine
(SVM) \cite{CorVap95}. As a consequence, this is the case
(with $d_{\mathcal{G}} = 2$) if $\mathcal{G}$ is the class of functions
computed by a multi-class SVM
\cite{Gue12,LeiDogBinKlo15,DogGlaIge16}.

\begin{theorem}
\label{theorem:dependency-on-m-and-gamma-L_2-norm}
Let $\mathcal{G}$ be a class of functions
satisfying Hypothesis~\ref{hypothesis:restriction-gamma-dim}.
For $\gamma \in \left ( 0, 1 \right ]$,
let $\mathcal{F}_{\mathcal{G}, \gamma}$ be the class of functions deduced
from $\mathcal{G}$ according to
Definition~\ref{def:class-of-regret-functions}.

\noindent If $d_{\mathcal{G}} = 1$, then
\begin{equation}
\label{eq:bound-on-Rademacher-complexity-1}
R_m \left ( \mathcal{F}_{\mathcal{G}, \gamma} \right ) 
\leqslant 
160 \sqrt{\frac{30 K_{\mathcal{G}} \gamma}{m}} C^{\frac{3}{4}}
\left [
\sqrt{\frac{\ln \left ( F \left ( C \right ) \right )}{2}}
+ \sqrt{\frac{\pi}{8}} F \left ( C \right )
\left ( 1 - \mbox{erf} \left ( \sqrt{\ln \left ( F \left ( C \right ) \right )}
\right ) \right )
\right ],
\end{equation}
where
$$
F \left ( C \right ) = 2 \sqrt{\frac{14 M_{\mathcal{G}}}{\gamma}} 
C^{\frac{1}{4}}
$$
and $\mbox{erf}$ stands for the {\em error function}, i.e.,
$ \mbox{erf} \left ( t \right ) = \frac{2}{\sqrt{\pi}}
\int_0^t e^{-u^2} \; du$.

\noindent If $d_{\mathcal{G}} = 2$, then
$$
R_m \left ( \mathcal{F}_{\mathcal{G}, \gamma} \right ) 
\leqslant \frac{\gamma C^{\frac{3}{4}}}{\sqrt{m}}
+ 1152 \sqrt{\frac{5 K_{\mathcal{G}}}{m}} C
\left \lceil \frac{1}{2} \log_2 \left ( \frac{m}{C} \right ) \right \rceil
\sqrt{\ln \left ( \frac{14 M_{\mathcal{G}} \sqrt{m}}
{\gamma C^{\frac{1}{4}}} \right )}.
$$
At last, if $d_{\mathcal{G}} > 2$, then
\begin{equation}
\label{eq:bound-on-Rademacher-complexity-3+}
R_m \left ( \mathcal{F}_{\mathcal{G}, \gamma} \right ) 
\leqslant \sqrt{C} \left (
\gamma \left ( \frac{C}{m} \right )^{\frac{1}{d_{\mathcal{G}}}}
+ 8 \cdot 96^{\frac{d_{\mathcal{G}}}{2}}
\left ( 2^{\frac{2}{d_{\mathcal{G}} - 2}} + 1 \right )
\cdot \gamma^{1 - \frac{d_{\mathcal{G}}}{2}}
\sqrt{5 K_{\mathcal{G}}} 
\left ( \frac{C}{m} \right )^{\frac{1}{d_{\mathcal{G}}}}
\sqrt{
\ln \left ( \frac{14 M_{\mathcal{G}}}{\gamma}
\left ( \frac{m}{C} \right )^{\frac{1}{d_{\mathcal{G}}}}
\right )} \right ).
\end{equation}
\end{theorem}

\begin{proof}
A substitution of Inequality~(\ref{eq:gamma-dim-bound}) into
Inequality~(\ref{eq:chaining-sum-final}) provides:
\begin{equation}
\label{eq:bound-on-Rademacher-complexity-partial-1}
R_m \left ( \mathcal{F}_{\mathcal{G}, \gamma} \right ) \leqslant
h \left ( N \right )
+ 4 \cdot 96^{\frac{d_{\mathcal{G}}}{2}}
\sqrt{\frac{5 K_{\mathcal{G}}}{m}} 
C^{\frac{d_{\mathcal{G}}+2 }{4}}
\sum_{j=1}^N 
\frac{h \left ( j \right ) + h \left ( j-1 \right )}
{h \left ( j \right )^{\frac{d_{\mathcal{G}}}{2}}}
\sqrt{
\ln \left ( \frac{14 M_{\mathcal{G}} \sqrt{C}}{h \left ( j \right )} \right )}.
\end{equation}
At this point, we distinguish three cases according to the value taken by
$d_{\mathcal{G}}$.

\paragraph{First case: $d_{\mathcal{G}} = 1$}
This case is the only one for which the entropy integral of
Formula~(\ref{eq:Dudley's-integral-inequality}) exists.
Setting for all $j \in \mathbb{N}$, 
$h \left ( j \right ) = \gamma \cdot 2^{-2j}$,
we obtain
\begin{equation}
\label{eq:bound-on-Rademacher-complexity-1-partial-1}
R_m \left ( \mathcal{F}_{\mathcal{G}, \gamma} \right )
\leqslant
160 \sqrt{\frac{30 K_{\mathcal{G}} \gamma}{m}} C^{\frac{3}{4}}
\int_0^\frac{1}{2} \sqrt{
\ln \left ( \frac{14 M_{\mathcal{G}} \sqrt{C}}
{\gamma \epsilon^2} \right )} \; d \epsilon.
\end{equation}
The computation of the integral gives
\begin{equation}
\label{eq:bound-on-Rademacher-complexity-1-partial-2}
\int_0^\frac{1}{2} \sqrt{
\ln \left ( \frac{14 M_{\mathcal{G}} \sqrt{C}}
{\gamma \epsilon^2} \right )} \; d \epsilon 
= \sqrt{\frac{\ln \left ( F \left ( C \right ) \right )}{2}}
+ \frac{F \left ( C \right )}{\sqrt{2}} \frac{\sqrt{\pi}}{2}
\left ( 1 - \mbox{erf} \left ( \sqrt{\ln \left ( F \left ( C \right ) \right )}
\right ) \right ).
\end{equation}
Inequality (\ref{eq:bound-on-Rademacher-complexity-1}) then results
from a substitution of the right-hand side of
(\ref{eq:bound-on-Rademacher-complexity-1-partial-2})
into (\ref{eq:bound-on-Rademacher-complexity-1-partial-1}).

\paragraph{Second case: $d_{\mathcal{G}} = 2$}

It stems from (\ref{eq:bound-on-Rademacher-complexity-partial-1}) that
$$
R_m \left ( \mathcal{F}_{\mathcal{G}, \gamma} \right ) \leqslant
h \left ( N \right )
+ 384 \sqrt{\frac{5 K_{\mathcal{G}}}{m}} C
\sum_{j=1}^N 
\frac{h \left ( j \right ) + h \left ( j-1 \right )}{h \left ( j \right )}
\sqrt{
\ln \left ( \frac{14 M_{\mathcal{G}} \sqrt{C}}{h \left ( j \right )} \right )}.
$$
For 
$N 
= \left \lceil \frac{1}{2} \log_2 \left ( \frac{m}{C} \right ) \right \rceil$,
we set
$h \left ( j \right ) = \gamma C^{\frac{3}{4}} m^{-\frac{1}{2}} 2^{-j+N}$.
Note that these choices are feasible since $N \in \mathbb{N}^*$
due to $m > C$,
$h \left ( 0 \right ) \geqslant \gamma C^{\frac{1}{4}} > \gamma$,
and 
$h \left ( 0 \right ) < 2 \gamma C^{\frac{1}{4}} < 2 M_{\mathcal{G}} \sqrt{C}$.
Then,
\begin{align*}
R_m \left ( \mathcal{F}_{\mathcal{G}, \gamma} \right ) 
& \leqslant \;
\frac{\gamma C^{\frac{3}{4}}}{\sqrt{m}}
+ 1152 \sqrt{\frac{5 K_{\mathcal{G}}}{m}} C
\sum_{j=1}^N 
\sqrt{
\ln \left ( \frac{14 M_{\mathcal{G}} \sqrt{m} \cdot 2^{j-N}}
{\gamma C^{\frac{1}{4}}} \right )} \\
& \leqslant \;
\frac{\gamma C^{\frac{3}{4}}}{\sqrt{m}}
+ 1152 \sqrt{\frac{5 K_{\mathcal{G}}}{m}} C
\left \lceil \frac{1}{2} \log_2 \left ( \frac{m}{C} \right ) \right \rceil
\sqrt{ \ln \left ( \frac{14 M_{\mathcal{G}} \sqrt{m}}
{\gamma C^{\frac{1}{4}}} \right )}.
\end{align*}

\paragraph{Third case: $d_{\mathcal{G}} > 2$}
For $N = \left \lceil \frac{d_{\mathcal{G}} - 2}{2 d_{\mathcal{G}}} 
\log_2 \left (
\frac{m}{C} \right ) \right \rceil$, let us set
$h \left ( j \right ) = \gamma 
C^{\frac{1}{2} + \frac{1}{d_{\mathcal{G}}}} 
m^{-\frac{1}{d_{\mathcal{G}}}} 
2^{\frac{2}{d_{\mathcal{G}} - 2} \left ( -j+N \right )}$.
Obviously, the constraints on $N$ and the function $h$
are once more satisfied. By substitution into
(\ref{eq:bound-on-Rademacher-complexity-partial-1}), we get:
\begin{equation}
\label{eq:partial-bound-on-Rademacher-complexity-3+}
R_m \left ( \mathcal{F}_{\mathcal{G}, \gamma} \right ) 
\leqslant
\sqrt{C} \left (
\gamma \left ( \frac{C}{m} \right )^{\frac{1}{d_{\mathcal{G}}}}
+ 4 \cdot 96^{\frac{d_{\mathcal{G}}}{2}}
\cdot \gamma^{1 - \frac{d_{\mathcal{G}}}{2}}
\sqrt{5 K_{\mathcal{G}}} 
\left ( \frac{C}{m} \right )^{\frac{1}{d_{\mathcal{G}}}}
\sqrt{
\ln \left ( \frac{14 M_{\mathcal{G}}}{\gamma}
\left ( \frac{m}{C} \right )^{\frac{1}{d_{\mathcal{G}}}} \right )}
S_N \right )
\end{equation}
with 
$$
S_N = \sum_{j=1}^N 
\frac{2^{\frac{2}{d_{\mathcal{G}} - 2} \left ( -j+N \right )}
+ 2^{\frac{2}{d_{\mathcal{G}} - 2} \left ( -j+1+N \right )}}
{2^{\frac{d_{\mathcal{G}}}{d_{\mathcal{G}} - 2} \left ( -j+N \right )}}.
$$
Now,
\begin{align*}
S_N
& = \;
\left ( 2^{\frac{2}{d_{\mathcal{G}} - 2}} + 1 \right )
\sum_{j=1}^N 2^{j-N} \\
& < \;
2 \left ( 2^{\frac{2}{d_{\mathcal{G}} - 2}} + 1 \right ).
\end{align*}
Inequality~(\ref{eq:bound-on-Rademacher-complexity-3+}) results from
a substitution of this upper bound on $S_N$ into
(\ref{eq:partial-bound-on-Rademacher-complexity-3+}).
\end{proof}

\subsection{Discussion}

The implementation of Dudley's chaining method 
under Hypothesis~\ref{hypothesis:restriction-gamma-dim} highlights 
the {\em phase transition} already identified by Mendelson in \cite{Men02}
(see also \cite{Men01}). Besides this well-known phenomenon regarding
the convergence rate, a parallel one can be noticed regarding 
the dependency on $C$. Indeed, if this dependency is always sublinear,
as announced, it varies significantly between $\sqrt{C}$ and $C$,
as a function of the value of $d_{\mathcal{G}}$.
Its asymptotic value is $\sqrt{C}$.
It is noteworthy that the behaviours observed are highly sensitive to
the choice of the function $h$. 
We have already noticed in the beginning of the section that
setting $h \left ( j \right ) = 2^{-j} \sqrt{C} \gamma$
has for consequence that the dependency on $C$ is uniformly linear.
Another example is instructive.
In the case $d_{\mathcal{G}} = 1$, choosing
$h \left ( j \right ) = \gamma \cdot 2^{-j}$ leads to
$$
R_m \left ( \mathcal{F}_{\mathcal{G}, \gamma} \right ) 
\leqslant 
96 \sqrt{\frac{30 K_{\mathcal{G}}}{m}} C \int_0^{\frac{\gamma}{2 \sqrt{C}}}
\sqrt{ \frac{1}{\epsilon}
\ln \left ( \frac{14 M_{\mathcal{G}}}{\epsilon} \right )} \; 
d \epsilon.
$$

\section{Conclusions and ongoing research}
\label{sec:conclusions}

An $L_p$-norm Sauer-Shelah lemma dedicated to margin multi-category classifiers
whose classes of component functions are uniform Glivenko-Cantelli classes
has been established. 
Its use makes it possible to improve the dependency on the number $C$
of categories of the state-of-the-art guaranteed risks based on
the $L_{\infty}$-norm and the $L_2$-norm.
In both cases, this dependency becomes sublinear.
Furthermore, in the favourable cases, the confidence interval can grow
with $C$ as slowly as a $O \left ( \sqrt{C} \right )$.

Our current work consists in continuing the unification of the approaches used
to derive the bounds with respect to the different $L_p$-norms.
The aim is to make the comparison of the resulting guaranteed risks
more straightforward, as a step towards the characterization of
the intrinsic complexity of the computation of polytomies.
We also look for improvements resulting from the use of new tools
from the theory of empirical processes.
In that respect, the recent developments of the implementation
of the chaining method appear promising.

Our results have been established under minimal assumptions regarding
the pattern classification problem, the classifier and the margin loss function.
Our future work will consist in assessing the benefit
that one can derive from this study under different assumptions,
such as those made in \cite{LeiDogBinKlo15}.

\acks{The author would like to thank R.~Vershynin for his explanations
on the proof of Theorem~1 in \cite{MenVer03} and A.~Kontorovich
for bringing to his attention the bounds in \cite{MohRosTal12}.
Thanks are also due to F.~Lauer and K.~Musayeva
for carefully reading this manuscript.
This work was partly funded by a CNRS research grant.}

\appendix

\section{Basic results and technical lemmas}

Our formulation of Dudley's metric entropy bound,
tailored for our needs, generalizes
that established in transcripts of Bartlett's lectures which can
be found online (see also \cite{AudBou07}). The integral inequality
appears as an instance of Corollary~13.2 in \cite{BouLugMas13}.

\begin{theorem}[Dudley's metric entropy bound]
\label{theorem:Dudley's-metric-entropy-bound}
Let $\mathcal{F}$ be a class of bounded
real-valued functions on $\mathcal{T}$.
For $n \in \mathbb{N}^*$, let
$\mathbf{t}_n = 
\left( t_i \right)_{1 \leqslant i \leqslant n} \in \mathcal{T}^n$
and let $\text{diam} \left ( \mathcal{F} \right ) =
\sup_{\left ( f, f' \right ) \in \mathcal{F}^2} 
\left \| f - f' \right \|_{L_2 \left ( \mu_{\mathbf{t}_n} \right )}$
be the diameter of $\mathcal{F}$ in the
$L_2 \left ( \mu_{\mathbf{t}_n} \right )$ seminorm.
Let $h$ be a positive and decreasing function on $\mathbb{N}$ such that
$h \left ( 0 \right ) \geqslant \text{diam} \left ( \mathcal{F} \right )$.
Then for $N \in \mathbb{N}^*$,
\begin{equation}
\label{eq:generalized-Dudley-metric-entropy-bound}
\hat{R}_n \left ( \mathcal{F} \right ) \leqslant h \left ( N \right )
+ 2 \sum_{j=1}^N \left ( h \left ( j \right ) + h \left ( j-1 \right ) \right )
\sqrt{ \frac{\ln \left (  
\mathcal{N}^{(p)} \left ( h \left ( j \right ), \mathcal{F}, 
d_{2, \mathbf{t}_n} \right )
\right )}{n}}
\end{equation}
and
\begin{equation}
\label{eq:Dudley's-integral-inequality}
\hat{R}_n \left ( \mathcal{F} \right ) \leqslant
12 \int_0^{\frac{1}{2} \cdot \text{diam} \left ( \mathcal{F} \right )}
\sqrt{\frac{ \ln \left ( 
\mathcal{N}^{(p)} \left ( \epsilon, \mathcal{F},
d_{2, \mathbf{t}_n} \right ) \right )}{n}} \; d \epsilon.
\end{equation}
\end{theorem}

\begin{proof}
For $j \in \mathbb{N}^*$,
let $\bar{\mathcal{F}}_j$ be a proper $h \left ( j \right )$-net 
of $\mathcal{F}$ with respect to $d_{2, \mathbf{t}_n}$ such that
$\left | \bar{\mathcal{F}}_j \right | = 
\mathcal{N}^{(p)} \left ( h \left ( j \right ), \mathcal{F}, 
d_{2, \mathbf{t}_n} \right )$.
We set $\bar{\mathcal{F}}_0 = \left \{ \bar{f}_0 \right \}$
where $\bar{f}_0$ is any function in $\mathcal{F}$.
Note that since $h \left ( 0 \right )$ can be equal to
$\text{diam} \left ( \mathcal{F} \right )$, the construction of
$\bar{\mathcal{F}}_0$ does not ensure that this set is
a proper $h \left ( 0 \right )$-net
of $\mathcal{F}$ with respect to $d_{2, \mathbf{t}_n}$
(the minimum cardinality of such a net can be superior or equal to $2$).
The Rademacher process underlying the Rademacher complexity is centered, i.e.,
$$
\forall f \in \mathcal{F}, \;\; \mathbb{E}_{\boldsymbol{\sigma}_n} \left [ 
\frac{1}{\sqrt{n}} \sum_{i=1}^n \sigma_i f \left ( t_i \right )
\right ] = 0.
$$
Thus,
$$
\hat{R}_n \left ( \mathcal{F} \right ) = 
\mathbb{E}_{\boldsymbol{\sigma}_n}
\left [ \sup_{f \in \mathcal{F}} \frac{1}{n}
\sum_{i=1}^n \sigma_i \left ( f \left ( t_i \right ) 
- \bar{f}_0 \left ( t_i \right ) \right ) \right ].
$$
For each $f \in \mathcal{F}$ and each $j \in \mathbb{N}^*$,
choose $\bar{f}_j \in \bar{\mathcal{F}}_j$ such that
$\left \| f - \bar{f}_j \right \|_{L_2 \left ( \mu_{\mathbf{t}_n} \right )} 
< h \left ( j \right )$. Notice that
$$
f - \bar{f}_0 = 
f - \bar{f}_N + \sum_{j=1}^N \left ( \bar{f}_j - \bar{f}_{j-1} \right ).
$$
As a consequence, making use of the sub-additivity of the supremum function
provides us with:

\begin{equation}
\label{eq:generalized-Dudley-partial-1}
\hat{R}_n \left ( \mathcal{F} \right ) \leqslant
\mathbb{E}_{\boldsymbol{\sigma}_n}
\left [ \sup_{f \in \mathcal{F}} \frac{1}{n}
\sum_{i=1}^n \sigma_i \left ( f \left ( t_i \right ) 
- \bar{f}_N \left ( t_i \right ) \right ) \right ]
+ \sum_{j=1}^N \mathbb{E}_{\boldsymbol{\sigma}_n}
\left [ \sup_{f \in \mathcal{F}} \frac{1}{n}
\sum_{i=1}^n \sigma_i \left ( \bar{f}_j \left ( t_i \right ) 
- \bar{f}_{j-1} \left ( t_i \right ) \right ) \right ].
\end{equation}
To bound from above the first term of the right-hand side of
(\ref{eq:generalized-Dudley-partial-1}),
we can make use in sequence of the Cauchy-Schwarz inequality and 
the definition of $h$.

\begin{align}
\mathbb{E}_{\boldsymbol{\sigma}_n}
\left [ \sup_{f \in \mathcal{F}} \frac{1}{n}
\sum_{i=1}^n \sigma_i \left ( f \left ( t_i \right ) 
- \bar{f}_N \left ( t_i \right ) \right ) \right ]
& \leqslant \; \mathbb{E}_{\boldsymbol{\sigma}_n}
\left [ \sup_{f \in \mathcal{F}} \left \{
\left ( \frac{1}{n} \sum_{i=1}^n \sigma_i^2 \right )^{\frac{1}{2}}
\left ( \frac{1}{n} \sum_{i=1}^n \left ( f \left ( t_i \right )
- \bar{f}_N \left ( t_i \right ) \right )^2 \right )^{\frac{1}{2}}
\right \} \right ] \nonumber \\
& \leqslant \; \sup_{f \in \mathcal{F}}
\left \| f - \bar{f}_N \right \|_{L_2 \left ( \mu_{\mathbf{t}_n} \right )}
\mathbb{E}_{\boldsymbol{\sigma}_n}
\left [ \left ( \frac{1}{n} \sum_{i=1}^n \sigma_i^2 \right )^{\frac{1}{2}}
\right ] \nonumber \\
\label{eq:generalized-Dudley-partial-2}
& < \; h \left ( N \right ).
\end{align}
As for the second term of the right-hand side of 
(\ref{eq:generalized-Dudley-partial-1}),
we make use of Massart's finite class lemma
(Lemma~5.2 in \cite{Mas00}). This calls for
the derivation of an upper bound on
$ \left \| \frac{1}{n} \left ( \bar{f}_j \left ( t_i \right ) - 
\bar{f}_{j-1} \left ( t_i \right ) \right )_{1 \leqslant i \leqslant n} 
\right \|_2 =
\frac{1}{\sqrt{n}} \left \| \bar{f}_j - \bar{f}_{j-1} 
\right \|_{L_2 \left ( \mu_{\mathbf{t}_n} \right )}$
for all $j \in \ieg 1, N \ied$. This upper bound is obtained by
application of Minkowski's inequality:

\begin{align}
\left \| \bar{f}_j - \bar{f}_{j-1} 
\right \|_{L_2 \left ( \mu_{\mathbf{t}_n} \right )}
& = \; \left \| \bar{f}_j - f + f - \bar{f}_{j-1} 
\right \|_{L_2 \left ( \mu_{\mathbf{t}_n} \right )} \nonumber \\
& \leqslant \; \left \| \bar{f}_j - f
\right \|_{L_2 \left ( \mu_{\mathbf{t}_n} \right )} + \left \| f - \bar{f}_{j-1}
\right \|_{L_2 \left ( \mu_{\mathbf{t}_n} \right )}
\nonumber \\
\label{eq:generalized-Dudley-partial-3}
& < h \left ( j \right ) + h \left ( j-1 \right ).
\end{align}
We can check that (\ref{eq:generalized-Dudley-partial-3})
still holds for $j=1$ since
$$
\begin{cases}
\left \| \bar{f}_1 - f 
\right \|_{L_2 \left ( \mu_{\mathbf{t}_n} \right )} < h \left ( 1 \right ) \\
\left \| f - \bar{f}_0
\right \|_{L_2 \left ( \mu_{\mathbf{t}_n} \right )} 
\leqslant \text{diam} \left ( \mathcal{F} \right )
\leqslant h \left ( 0 \right )
\end{cases}
\Longrightarrow
\left \| \bar{f}_1 - f \right \|_{L_2 \left ( \mu_{\mathbf{t}_n} \right )}
+ \left \| f - \bar{f}_0 \right \|_{L_2 \left ( \mu_{\mathbf{t}_n} \right )}
< h \left ( 1 \right ) + h \left ( 0 \right ).
$$
Applying Lemma~5.2 in \cite{Mas00} with
(\ref{eq:generalized-Dudley-partial-3}) gives:
\begin{align}
\forall j \in \ieg 1, N \ied, \;\;
\mathbb{E}_{\boldsymbol{\sigma}_n}
\left [ \sup_{f \in \mathcal{F}} \frac{1}{n}
\sum_{i=1}^n \sigma_i \left ( \bar{f}_j \left ( t_i \right ) 
- \bar{f}_{j-1} \left ( t_i \right ) \right ) \right ]
& \leqslant \; \frac{h \left ( j \right ) + h \left ( j-1 \right )}{\sqrt{n}}
\sqrt{ 2 \ln \left ( \left | \bar{\mathcal{F}}_j \right |
\left | \bar{\mathcal{F}}_{j-1} \right | \right )} \nonumber \\
\label{eq:generalized-Dudley-partial-4}
& \leqslant \; 2 \left ( h \left ( j \right ) + h \left ( j-1 \right ) \right )
\sqrt{ \frac{\ln \left ( \left | \bar{\mathcal{F}}_j \right | \right )}{n}}.
\end{align}
The substitution of (\ref{eq:generalized-Dudley-partial-2}) and
(\ref{eq:generalized-Dudley-partial-4}) into
(\ref{eq:generalized-Dudley-partial-1}) produces
(\ref{eq:generalized-Dudley-metric-entropy-bound}). Furthermore,
setting for all $j \in \mathbb{N}$, $h \left ( j \right ) = 2^{-j} \cdot
\text{diam} \left ( \mathcal{F} \right )$, gives:
\begin{align}
\hat{R}_n \left ( \mathcal{F} \right ) & \leqslant \;
\text{diam} \left ( \mathcal{F} \right ) \left (
2^{-N} + 6 \sum_{j=1}^N 2^{-j} \sqrt{\frac{ \ln \left ( 
\mathcal{N}^{(p)} \left ( 2^{-j} \cdot \text{diam} \left ( \mathcal{F} \right ),
\mathcal{F}, d_{2, \mathbf{t}_n} \right ) \right )}{n}} \right ) \nonumber \\
\label{eq:Dudley-partial-a2}
& \leqslant \;
\text{diam} \left ( \mathcal{F} \right ) \left (
2^{-N} + 12 \sum_{j=1}^N \left ( 2^{-j} - 2^{- \left (j + 1 \right )} \right )
\sqrt{\frac{ \ln \left ( 
\mathcal{N}^{(p)} \left ( 2^{-j} \cdot \text{diam} \left ( \mathcal{F} \right ),
\mathcal{F}, d_{2, \mathbf{t}_n} \right ) \right )}{n}} \right ) \\
\label{eq:Dudley-partial-a3}
& \leqslant \; 
2^{-N} \cdot \text{diam} \left ( \mathcal{F} \right )
+ 12 \int_{\frac{1}{2^{N+1}} \cdot \text{diam} \left ( \mathcal{F} \right )
}^{\frac{1}{2} \cdot \text{diam} \left ( \mathcal{F} \right )}
\sqrt{\frac{ \ln \left ( 
\mathcal{N}^{(p)} \left ( \epsilon, \mathcal{F},
d_{2, \mathbf{t}_n} \right ) \right )}{n}} \; d \epsilon.
\end{align}
Inequality~(\ref{eq:Dudley-partial-a3}) springs from
Inequality~(\ref{eq:Dudley-partial-a2}) since a covering number is
a nonincreasing function of $\epsilon$
(on the interval 
$\left [ 2^{- \left (j + 1 \right )} 
\cdot \text{diam} \left ( \mathcal{F} \right ), 
2^{-j} \cdot \text{diam} \left ( \mathcal{F} \right ) \right ]$,
$\mathcal{N}^{(p)} \left ( 
2^{-j} \cdot \text{diam} \left ( \mathcal{F} \right ),
\mathcal{F}, d_{2, \mathbf{t}_n} \right ) \leqslant
\mathcal{N}^{(p)} \left ( \epsilon, \mathcal{F}, d_{2, \mathbf{t}_n} \right )$).
Inequality~(\ref{eq:Dudley's-integral-inequality}) is simply the asymptotic
formulation of Inequality~(\ref{eq:Dudley-partial-a3}) 
(for $N$ going to infinity).
\end{proof}

\begin{lemma}[After Theorem~IV in \cite{KolTih61}]
\label{lemma:Kolmogorov}
Let $\left (E, \rho \right )$ be a pseudo-metric space.
For every totally bounded set $E' \subset E$ and
$\epsilon \in \mathbb{R}_+^*$,
$$
\mathcal{N}^{(p)} \left ( \epsilon, E', \rho \right ) \leqslant
\mathcal{M} \left ( \epsilon, E', \rho \right ).
$$
\end{lemma}

\begin{lemma}[After Lemma~13 in \cite{MenVer03}]
\label{lemma:lemma-13-in-MenVer03-in-L_p}
Let $\mathcal{T} = \left \{ t_i: \; 1 \leqslant i \leqslant n \right \}$
be a finite set
and $\mathbf{t}_n = \left( t_i \right)_{1 \leqslant i \leqslant n}$.
Let $\mathcal{F}$ be a finite class of functions from $\mathcal{T}$
into $\left [ -M_{\mathcal{F}}, M_{\mathcal{F}} \right ]$
with $M_{\mathcal{F}} \in \mathbb{R}_+^*$. Let $p \in \mathbb{N}^*$.
Assume that for some $\epsilon \in \left ( 0, 2M_{\mathcal{F}} \right ]$,
$\mathcal{F}$ is $\epsilon$-separated
with respect to the pseudo-metric $d_{p, \mathbf{t}_n}$.
If $r \in \left [ 1, n \right ]$ is such that
$\left | \mathcal{F} \right | \leqslant
\exp \left ( K_e \left ( p \right ) r \epsilon^{2p} \right )$
with
$$
K_e \left ( p \right ) = \frac{3}
{112 \left ( 2 M_{\mathcal{F}} \right )^{2p}},
$$
then there exists a subvector $\mathbf{t}_q$
of $\mathbf{t}_n$ of size $q \leqslant r$ such that
$\mathcal{F}$ is 
$\left ( \left ( \frac{1}{2} \right )^{\frac{p+1}{p}}
\epsilon \right )$-separated with respect to
the pseudo-metric $d_{p, \mathbf{t}_q}$.
\end{lemma}

\begin{proof}
Let us set $\mathcal{F} = \left \{ f_j: \; 1 \leqslant j \leqslant 
\left | \mathcal{F} \right | \right \}$
and $\mathcal{D}_{\mathcal{F}} = \left \{ 
f_j - f_{j'}: \; 1 \leqslant j < j'  \leqslant \left | \mathcal{F} \right |
\right \}$. The set $\mathcal{D}_{\mathcal{F}}$
has cardinality $\left | \mathcal{D}_{\mathcal{F}} \right | <
\frac{1}{2} \left | \mathcal{F} \right |^2$.
Fix $r \in \left [1, n \right ]$
satisfying the assumptions of the lemma and let
$\left ( \epsilon_i \right )_{1 \leqslant i \leqslant n}$ be a sequence
of $n$ independent Bernoulli random variables with common expectation
$\mu = \frac{r}{2n}$.
Then, by application of the $\epsilon$-separation property,
for every $\delta_f$ in $\mathcal{D}_{\mathcal{F}}$,
\begin{align}
\mathbb{P}
\left ( \frac{1}{n} \sum_{i=1}^n \epsilon_i 
\left | \delta_f \left ( t_i \right ) \right |^p
< \mu \left ( \frac{\epsilon}{2} \right )^p \right )
& \leqslant \;
\mathbb{P} \left ( \frac{1}{n}
\sum_{i=1}^n \left ( \mu - \epsilon_i \right )
\left | \delta_f \left ( t_i \right ) \right |^p
> \left ( 1 - \frac{1}{2^p} \right ) \mu \epsilon^p 
\right ) \nonumber \\
\label{eq:partial1-lemma-13-in-MenVer03-in-L_p}
& \leqslant \; 
\mathbb{P} \left ( \frac{1}{n}
\sum_{i=1}^n \left ( \mu - \epsilon_i \right )
\left | \delta_f \left ( t_i \right ) \right |^p
> \frac{1}{2} \mu \epsilon^p 
\right ).
\end{align}
Since by construction, for all $i \in \ieg 1, n \ied$,
$\mathbb{E} \left [
\left ( \mu - \epsilon_i \right ) 
\left | \delta_f \left ( t_i \right ) \right |^p \right ] = 0$
and $\left | \mu - \epsilon_i \right |
\left | \delta_f \left ( t_i \right ) \right |^p
\leqslant \left ( 2 M_{\mathcal{F}} \right )^p \left ( 1 - \mu \right )
< \left ( 2 M_{\mathcal{F}} \right )^p$ 
with probability one,
the right-hand side of (\ref{eq:partial1-lemma-13-in-MenVer03-in-L_p}) can be
bounded from above thanks to Bernstein's inequality \cite{Ber46}.
Given that
$$
\frac{1}{n} \sum_{i=1}^n \mathbb{E} \left [ \left ( \mu - \epsilon_i \right )^2
\delta_f \left ( t_i \right )^{2p} \right ]
\leqslant \left ( 2 M_{\mathcal{F}} \right )^{2p} \mu \left ( 1 - \mu \right )
< \left ( 2 M_{\mathcal{F}} \right )^{2p} \mu,
$$
we obtain
\begin{align*}
\mathbb{P}
\left ( \frac{1}{n} \sum_{i=1}^n \epsilon_i 
\left | \delta_f \left ( t_i \right ) \right |^p
< \mu \left ( \frac{\epsilon}{2} \right )^p \right )
& \leqslant \;
\exp \left ( - \frac{3 \mu n \epsilon^{2p}}
{4 \left ( 6 \left ( 2 M_{\mathcal{F}} \right )^{2p} + 
\left ( 2 M_{\mathcal{F}} \right )^p \epsilon^p \right )} \right ) \\
& \leqslant \; 
\exp \left ( - \frac{3 r \epsilon^{2p}}
{56 \left ( 2 M_{\mathcal{F}} \right )^{2p}} \right ) \\
& \leqslant \; 
\exp \left ( -2 K_e \left ( p \right ) r \epsilon^{2p} \right ).
\end{align*}
Therefore, given the assumption on $r$, applying
the union bound provides us with:
\begin{align}
\mathbb{P} \left ( \exists \delta_f \in \mathcal{D}_{\mathcal{F}}: \;
\left ( \frac{1}{r} \sum_{i=1}^n \epsilon_i 
\left | \delta_f \left ( t_i \right ) \right |^p
\right )^{\frac{1}{p}} < 
\left ( \frac{1}{2} \right )^{\frac{p+1}{p}} \epsilon \right )
& = \;
\mathbb{P} \left ( \exists \delta_f \in \mathcal{D}_{\mathcal{F}}: \;
\frac{1}{n} \sum_{i=1}^n \epsilon_i
\left | \delta_f \left ( t_i \right ) \right |^p
< \mu \left ( \frac{\epsilon}{2} \right )^p
\right ) \nonumber \\
& \leqslant \;
\sum_{\delta_f \in \mathcal{D}_{\mathcal{F}}} 
\mathbb{P} \left ( \frac{1}{n}
\sum_{i=1}^n \epsilon_i
\left | \delta_f \left ( t_i \right ) \right |^p
< \mu \left ( \frac{\epsilon}{2} \right )^p \right ) \nonumber \\
& \leqslant \;
\left | \mathcal{D}_{\mathcal{F}} \right |
\cdot \exp \left ( -2 K_e \left ( p \right ) r \epsilon^{2p} \right ) 
\nonumber \\
& < \;
\frac{1}{2} \exp^2 \left ( K_e \left ( p \right ) r \epsilon^{2p} \right )
\cdot \exp \left ( -2 K_e \left ( p \right ) r \epsilon^{2p} \right ) 
\nonumber \\
\label{eq:partial3-lemma-13-in-MenVer03-in-L_p}
& < \; \frac{1}{2}.
\end{align}
Moreover, if $\mathcal{S}_1$ is the random set
$\left \{ i \in \ieg 1, n \ied: \; \epsilon_i = 1 \right \}$,
then by Markov's inequality,
\begin{equation}
\label{eq:partial4-lemma-13-in-MenVer03-in-L_p}
\mathbb{P} \left ( \left | S_1 \right | > r \right ) =
\mathbb{P} \left ( \sum_{i=1}^n \epsilon_i > r \right ) \leqslant
\frac{1}{2}.
\end{equation}
Combining (\ref{eq:partial3-lemma-13-in-MenVer03-in-L_p})
and (\ref{eq:partial4-lemma-13-in-MenVer03-in-L_p})
by means of the union bound provides us with
$$
\mathbb{P}
\left \{ \left ( \exists \delta_f \in \mathcal{D}_{\mathcal{F}}: \;
\left ( \frac{1}{r} \sum_{i=1}^n \epsilon_i 
\left | \delta_f \left ( t_i \right ) \right |^{p}
\right )^{\frac{1}{p}}
< \left ( \frac{1}{2} \right )^{\frac{p+1}{p}} \epsilon \right ) 
\vee \left ( \left | S_1 \right | > r \right ) \right \} < 1
$$
or equivalently
$$
\mathbb{P}
\left \{ \left ( \forall \delta_f \in \mathcal{D}_{\mathcal{F}}: \;
\left ( \frac{1}{r} \sum_{i=1}^n \epsilon_i
\left | \delta_f \left ( t_i \right ) \right |^{p}
\right )^{\frac{1}{p}}
\geqslant \left ( \frac{1}{2} \right )^{\frac{p+1}{p}} \epsilon \right )
\wedge \left ( \left | S_1 \right | \leqslant r \right ) \right \} > 0
$$
which implies that
$$
\mathbb{P}
\left \{ \left ( \forall \delta_f \in \mathcal{D}_{\mathcal{F}}: \;
\left \| \delta_f \right \|_{L_p \left ( 
\mu_{\left ( t_i \right )_{i \in \mathcal{S}_1}} \right )}
\geqslant \left ( \frac{1}{2} \right )^{\frac{p+1}{p}} \epsilon \right )
\wedge \left ( \left | S_1 \right | \leqslant r \right ) \right \} > 0.
$$
This translates into the fact that there exists
a subvector $\mathbf{t}_q$
of $\mathbf{t}_n$ of size $q \leqslant r$
such that the class $\mathcal{F}$ is
$\left ( \left ( \frac{1}{2} \right )^{\frac{p+1}{p}}
\epsilon \right )$-separated
with respect to the pseudo-metric $d_{p, \mathbf{t}_q}$, i.e., our claim.
\end{proof}

\begin{lemma}
\label{lemma:lemma-3.2.2-in-AloBenCesHau97-in-L_p}
Let $\mathcal{F}$ be a class of functions from $\mathcal{T}$ into
$\left [ -M_{\mathcal{F}}, M_{\mathcal{F}} \right ]$ with
$M_{\mathcal{F}} \in \mathbb{R}_+^*$.
For $n \in \mathbb{N}^*$, let
$\mathbf{t}_n = \left( t_i \right)_{1 \leqslant i \leqslant n} 
\in \mathcal{T}^n$.
For all $\epsilon \in \left ( 0, 2 M_{\mathcal{F}}\right ]$,
all $\eta \in \left ( 0, \epsilon \right )$,
and all $p \in \mathbb{N}^*$,
if a subset of $\mathcal{F}$ is $\epsilon$-separated
in the pseudo-metric $d_{p, \mathbf{t}_n}$,
then the $\eta$-discretization operator acts on it as an injective mapping
and the image obtained is a set
$\left ( \frac{\left ( \epsilon^p - \eta^p \right )^{\frac{1}{p}}}{2} 
\right )$-separated
in the same pseudo-metric. As a consequence,
$$
\forall p \in \mathbb{N}^*, \;\;
\mathcal{M} \left ( \epsilon, \mathcal{F}, d_{p, \mathbf{t}_n} \right )
\leqslant
\mathcal{M} 
\left ( \frac{\left ( \epsilon^p - \eta^p \right )^{\frac{1}{p}}}{2},
\mathcal{F}^{(\eta)}, d_{p, \mathbf{t}_n} \right ).
$$
\end{lemma}

\begin{proof}
Proving Lemma~\ref{lemma:lemma-3.2.2-in-AloBenCesHau97-in-L_p}
amounts to establishing that
\begin{equation}
\label{eq:lemma-3.2.2-in-AloBenCesHau97_in_L_p}
\forall \left ( f_1, f_2 \right ) \in \mathcal{F}^2, \;\;
\left ( d_{p, \mathbf{t}_n} 
\left ( f_1, f_2 \right ) \geqslant \epsilon \right )
\wedge \left ( \eta \in \left ( 0, \epsilon \right ) \right )
\Longrightarrow
d_{p, \mathbf{t}_n} 
\left ( f_1^{(\eta)}, f_2^{(\eta)} \right ) 
\geqslant \frac{\left ( \epsilon^p - \eta^p \right )^{\frac{1}{p}}}{2}.
\end{equation}
For $i \in \ieg 1, n \ied$, let 
$\delta_i = \left ( f_1^{(\eta)} \left ( t_i \right ) 
- f_2^{(\eta)} \left ( t_i \right ) \right )$ and
$\delta'_i = f_1 \left ( t_i \right ) - f_2 \left ( t_i \right ) - \delta_i$.
By construction, there exists $N_i \in \mathbb{N}$ such that
$\left | \delta_i \right | = \eta N_i$, and
$\left | \delta'_i \right | < \eta$.
If $N_i > 0$, then $\left | \delta_i \right | + \left | \delta'_i \right |
< 2 \left | \delta_i \right |$,
otherwise $\left | \delta_i \right | + \left | \delta'_i \right | < \eta$,
with the consequence that in all cases,
$\left ( \left | \delta_i \right |
+ \left | \delta'_i \right | \right )^p 
< \left ( 2 \left | \delta_i \right | \right )^p + \eta^p$.
Thus,
\begin{align*}
\left (
d_{p, \mathbf{t}_n} \left ( f_1, f_2 \right ) \geqslant \epsilon \right )
\wedge \left ( \eta \in \left ( 0, \epsilon \right ) \right )
& \Longrightarrow \;
\frac{1}{n} \sum_{i=1}^n \left | \delta_i + \delta'_i \right |^p
\geqslant \epsilon^p \\
& \Longrightarrow \;
\frac{1}{n} \sum_{i=1}^n \left ( \left | \delta_i \right |
+ \left | \delta'_i \right | \right )^p
\geqslant \epsilon^p \\
& \Longrightarrow \;
\frac{1}{n} \sum_{i=1}^n \left ( 2 \left | \delta_i \right | \right )^p + \eta^p
\geqslant \epsilon^p \\
& \Longrightarrow \; 
\left ( 2 d_{p, \mathbf{t}_n} \left ( f_1^{(\eta)}, f_2^{(\eta)} \right )
\right )^p + \eta^p \geqslant \epsilon^p \\
& \Longrightarrow \;
d_{p, \mathbf{t}_n} \left ( f_1^{(\eta)}, f_2^{(\eta)} \right ) \geqslant 
\frac{\left ( \epsilon^p - \eta^p \right )^{\frac{1}{p}}}{2}.
\end{align*}
To sum up, we have established (\ref{eq:lemma-3.2.2-in-AloBenCesHau97_in_L_p}),
i.e., the lemma.
\end{proof}

\begin{lemma}
\label{lemma:lemma-3.2.1-in-AloBenCesHau97-in-L_p}
Let $\mathcal{F}$ be a class of functions from $\mathcal{T}$ into
$\left [ -M_{\mathcal{F}}, M_{\mathcal{F}} \right ]$ with
$M_{\mathcal{F}} \in \mathbb{R}_+^*$.
For all $\epsilon \in \left ( 0, M_{\mathcal{F}} \right ]$
and all $\eta \in \left ( 0, 2 \epsilon \right )$,
$$
\epsilon\mbox{-dim} \left ( \mathcal{F}^{(\eta)} \right )
\leqslant \left ( \epsilon - \frac{\eta}{2} \right )\mbox{-dim} 
\left ( \mathcal{F} \right ).
$$
\end{lemma}

\begin{proof}
To prove Lemma~\ref{lemma:lemma-3.2.1-in-AloBenCesHau97-in-L_p},
it suffices to notice that
\begin{align*}
f^{\left ( \eta \right )} \left ( t \right ) - b \geqslant \epsilon
& \Longrightarrow \;
\eta \left \lfloor \frac{f \left ( t \right )
+ M_{\mathcal{F}}}{\eta} \right \rfloor - b \geqslant \epsilon \\
& \Longrightarrow \;
f \left ( t \right ) + M_{\mathcal{F}} - b \geqslant \epsilon \\
& \Longrightarrow \;
f \left ( t \right ) - \left ( b + \frac{\eta}{2} - M_{\mathcal{F}} \right )
\geqslant \epsilon - \frac{\eta}{2}
\end{align*}
and
\begin{align*}
f^{\left ( \eta \right )} \left ( t \right ) - b \leqslant - \epsilon
& \Longrightarrow \;
\eta \left \lfloor \frac{f \left ( t \right )
+ M_{\mathcal{F}}}{\eta} \right \rfloor - b \leqslant - \epsilon \\
& \Longrightarrow \;
f \left ( t \right ) + M_{\mathcal{F}} - \eta - b \leqslant -\epsilon \\
& \Longrightarrow \;
f \left ( t \right ) - \left ( b + \frac{\eta}{2} - M_{\mathcal{F}} \right )
\leqslant - \left ( \epsilon - \frac{\eta}{2} \right ).
\end{align*}
\end{proof}
In the framework of this study, the main combinatorial result
evoqued in Section~\ref{sec:capacity-measures} is
the following lemma, which extends Lemma~8 in \cite{BarLon95}.

\begin{lemma}
\label{lemma:Lemma-8-in-BarLon95-in-L_p}
Let $\mathcal{T} = \left \{ t_i: 1 \leqslant i \leqslant n \right \}$
be a finite set and
$\mathbf{t}_n = \left( t_i \right)_{1 \leqslant i \leqslant n}$.
Let $\mathcal{F}$ be a class of functions
from $\mathcal{T}$ into 
$\mathcal{S} = \left \{ 2 M_{\mathcal{F}} \frac{j}{N}: 
0 \leqslant j \leqslant N \right \}$
with $M_{\mathcal{F}} \in \mathbb{R}_+^*$ and 
$N \in \mathbb{N} \setminus \ieg 0, 3 \ied$.
For $\epsilon \in \left ( \frac{6 M_{\mathcal{F}}}{N}, 
2 M_{\mathcal{F}} \right ]$, let
$d = \left ( \frac{\epsilon}{2} 
- \frac{3 M_{\mathcal{F}}}{N} \right )\mbox{-dim} 
\left ( \mathcal{F} \right )$. Then
\begin{equation}
\label{eq:Sauer-Shelah-L_p}
\forall p \in \mathbb{N}^*, \;\;
\mathcal{M} \left ( \epsilon, \mathcal{F}, d_{p, \mathbf{t}_n} \right ) 
< 2^{\left ( p+2 \right ) \log_2 \left ( N \right ) + 1}
\left ( \frac{e \left ( N-1 \right ) n}{d} 
\right )^{\left ( p+2 \right ) \log_2 \left ( N \right ) d}.
\end{equation}
\end{lemma}

\begin{proof}
First, note that
$$
\begin{cases}
d \geqslant 1 \\
\mathcal{M} \left ( \epsilon, \mathcal{F}, d_{p, \mathbf{t}_n} \right )
\geqslant 2
\end{cases}
\Longrightarrow
\begin{cases}
\epsilon \in \left ( \frac{6 M_{\mathcal{F}}}{N}, 
2 M_{\mathcal{F}} + \frac{6 M_{\mathcal{F}}}{N} \right ] \\
\epsilon \in \left ( 0, 2 M_{\mathcal{F}} \right ]
\end{cases}
\Longrightarrow
\begin{cases}
\epsilon \in \left ( \frac{6 M_{\mathcal{F}}}{N}, 
2 M_{\mathcal{F}} \right ] \\
N > 3
\end{cases}.
$$
For $q \in \ieg 1, d \ied$,
let the pair $\left ( s_{\mathcal{T}^q}, \mathbf{b}_q \right )$
be such that $s_{\mathcal{T}^q}$ is a subset of $\mathcal{T}$ of cardinality $q$
and $\mathbf{b}_q \in \left ( \mathcal{S} \setminus 
\left \{ 0, 2 M_{\mathcal{F}} \right \} \right )^q$.
Such a pair will be said to be $\gamma$-shattered by a subset of
$\mathcal{F}$ if $s_{\mathcal{T}^q}$ is $\gamma$-shattered by
this subset and $\mathbf{b}_q$ is a witness to this shattering.
Setting $K = \sum_{j=0}^d { n \choose j} \left ( N - 1 \right )^j$,
the number of such pairs is equal to
$\sum_{j=1}^d { n \choose j} \left ( N - 1 \right )^j$, i.e., to $K-1$.
Fix $\epsilon \in \left ( \frac{6 M_{\mathcal{F}}}{N}, 
2 M_{\mathcal{F}} \right ]$
and $p \in \mathbb{N}^*$.
For each $r \in \ieg 2, \mathcal{M} \left ( \epsilon, \mathcal{F}, 
d_{p, \mathbf{t}_n} \right ) \ied$,
let $\mbox{shat} \left ( r \right )$
be the maximum integer such that any subset of $\mathcal{F}$ of
cardinality $r$ which is $\epsilon$-separated in the metric
$d_{p, \mathbf{t}_n}$
$\left ( \frac{\epsilon}{2} - \frac{3 M_{\mathcal{F}}}{N} \right )$-shatters
at least $\mbox{shat} \left ( r \right )$ pairs
$\left ( s_{\mathcal{T}^q}, \mathbf{b}_q \right )$.
Obviously, the function $\mbox{shat}$ is nondecreasing.
We now establish that $\mbox{shat} \left ( 2 \right ) \geqslant 1$.
Indeed, let $\left \{ f_+, f_- \right \}$ be a subset of $\mathcal{F}$
$\epsilon$-separated in the metric $d_{p, \mathbf{t}_n}$.
By definition,
$$
\left ( \frac{1}{n} \sum_{i=1}^n \left | f_+  \left ( t_i \right )
- f_- \left ( t_i \right ) \right |^p \right )^{\frac{1}{p}}
\geqslant \epsilon,
$$
with the consequence that there exists $i_0 \in \ieg 1, n \ied$
such that $\left | f_+  \left ( t_{i_0} \right )
- f_- \left ( t_{i_0} \right ) \right | \geqslant \epsilon$.
Without loss of generality, we make the hypothesis that
$f_+  \left ( t_{i_0} \right ) - f_- \left ( t_{i_0} \right ) 
\geqslant \epsilon$. Then,
\begin{align*}
f_+ \left ( t_{i_0}  \right )
- \frac{2 M_{\mathcal{F}}}{N} \left \lfloor \frac{N}{4 M_{\mathcal{F}}} 
\left ( f_+  \left ( t_{i_0} \right ) + f_- \left ( t_{i_0} \right ) \right ) 
\right \rfloor
& \geqslant \; f_+ \left ( t_{i_0}  \right ) - \frac{1}{2}
\left ( f_+  \left ( t_{i_0} \right ) + f_- \left ( t_{i_0} \right ) \right ) \\
& \geqslant \; \frac{1}{2} \left ( f_+  \left ( t_{i_0} \right )
- f_- \left ( t_{i_0} \right ) \right ) \\
& \geqslant \; \frac{\epsilon}{2} \\
& \geqslant \; \frac{\epsilon}{2} - \frac{3 M_{\mathcal{F}}}{N}
\end{align*}
and
\begin{align*}
f_- \left ( t_{i_0}  \right )
- \frac{2 M_{\mathcal{F}}}{N} \left \lfloor \frac{N}{4 M_{\mathcal{F}}} 
\left ( f_+  \left ( t_{i_0} \right ) + f_- \left ( t_{i_0} \right ) \right ) 
\right \rfloor
& \leqslant \; 
f_- \left ( t_{i_0}  \right ) - \frac{1}{2}
\left ( f_+  \left ( t_{i_0} \right ) + f_- \left ( t_{i_0} \right ) \right ) 
+ \frac{2 M_{\mathcal{F}}}{N} \\
& \leqslant \;  \frac{1}{2}
\left ( f_-  \left ( t_{i_0} \right ) - f_+ \left ( t_{i_0} \right ) \right ) 
+ \frac{2 M_{\mathcal{F}}}{N} \\
& \leqslant \;   
- \frac{\epsilon}{2} + \frac{2 M_{\mathcal{F}}}{N} \\
& \leqslant \; 
- \left ( \frac{\epsilon}{2} - \frac{3 M_{\mathcal{F}}}{N} \right ).
\end{align*}
Since $\frac{2 M_{\mathcal{F}}}{N} 
\left \lfloor \frac{N}{4 M_{\mathcal{F}}} 
\left ( f_+  \left ( t_{i_0} \right ) + f_- \left ( t_{i_0} \right ) \right ) 
\right \rfloor 
\in \mathcal{S} \setminus \left \{ 0, 2 M_{\mathcal{F}} \right \}$,
we have established that the set $\left \{ f_+, f_- \right \}$
$\left ( \frac{\epsilon}{2} - \frac{3 M_{\mathcal{F}}}{N} \right )$-shatters
$\left ( \left \{ t_{i_0} \right \}, \mathbf{b}_1 \right )$ 
with $\mathbf{b}_1 = \left ( \frac{2 M_{\mathcal{F}}}{N} 
\left \lfloor \frac{N}{4 M_{\mathcal{F}}} 
\left ( f_+  \left ( t_{i_0} \right ) + f_- \left ( t_{i_0} \right ) \right ) 
\right \rfloor \right )$, which concludes the proof of
$\mbox{shat} \left ( 2 \right ) \geqslant 1$.
Choose an even $r \in \ieg 2, \mathcal{M} \left ( \epsilon, \mathcal{F}, 
d_{p, \mathbf{t}_n} \right ) \ied$ and
let $\bar{\mathcal{F}}$ be a subset of $\mathcal{F}$
of cardinality $r$
$\epsilon$-separated in the metric $d_{p, \mathbf{t}_n}$.
Split $\bar{\mathcal{F}}$ arbitrarily into $\frac{r}{2}$ pairs.
For each such pair $\left ( f_+, f_- \right )$, let
$$
\mbox{ind} \left ( f_+, f_- \right ) = \left |
\left \{ i \in \ieg 1, n \ied: \; \left | f_+ \left ( t_i \right ) 
- f_- \left ( t_i \right ) \right | \geqslant \epsilon 
- \frac{2 M_{\mathcal{F}}}{N} \right \}
\right |.
$$
Then
\begin{align*}
d_{p, \mathbf{t}_n}  \left ( f_+, f_- \right )
& = \; \left ( \frac{1}{n} \sum_{i=1}^n \left | f_+  \left ( t_i \right )
- f_- \left ( t_i \right ) \right |^p \right )^{\frac{1}{p}} \\
& \leqslant \; \left \{ \frac{1}{n} \left [
\mbox{ind} \left ( f_+, f_- \right ) \left ( 2 M_{\mathcal{F}} \right )^p
+ \left ( n - \mbox{ind} \left ( f_+, f_- \right ) \right )
\left ( \epsilon - \frac{2 M_{\mathcal{F}}}{N} \right )^p \right ] 
\right \}^{\frac{1}{p}} \\
& \leqslant \;
\left [ \frac{\mbox{ind} \left ( f_+, f_- \right )}{n} 
\left ( 2 M_{\mathcal{F}} \right )^p
+ \left ( \epsilon - \frac{2 M_{\mathcal{F}}}{N} \right )^p 
\right ]^{\frac{1}{p}}.
\end{align*}
By hypothesis,
$d_{p, \mathbf{t}_n}  \left ( f_+, f_- \right ) \geqslant \epsilon
> \frac{6 M_{\mathcal{F}}}{N}$,
which implies that 
\begin{align*}
\mbox{ind} \left ( f_+, f_- \right )
& \geqslant \;
\frac{n}{\left ( 2 M_{\mathcal{F}} \right )^p}
\left [ \epsilon^p - \left ( \epsilon - 
\frac{2 M_{\mathcal{F}}}{N} \right )^p \right ] \\
& \geqslant \; \frac{n}{N \left ( 2 M_{\mathcal{F}} \right )^{p-1}} 
\sum_{j=0}^{p-1} \epsilon^{p-j-1}
\left ( \epsilon - \frac{2 M_{\mathcal{F}}}{N} \right )^j \\
& \geqslant \; \frac{n}{N \left ( 2 M_{\mathcal{F}} \right )^{p-1}} 
\sum_{j=0}^{p-1} \left ( \frac{6 M_{\mathcal{F}}}{N} \right )^{p-j-1}
\left ( \frac{4 M_{\mathcal{F}}}{N} \right )^j \\
& \geqslant \; \frac{n}{N^p} \sum_{j=0}^{p-1} 3^{p-j-1} 2^j \\
& \geqslant \; \frac{3^p - 2^p}{N^p} n \\
& \geqslant \; \frac{n}{N^p}.
\end{align*}
Thus, each pair $\left ( f_+, f_- \right )$
has at least $\frac{n}{N^p}$ indices $i$ such that
$\left | f_+  \left ( t_i \right ) - f_- \left ( t_i \right ) \right |
\geqslant \epsilon - \frac{2 M_{\mathcal{F}}}{N}$. 
Applying the pigeonhole principle,
there is at least one index $i_0$ such that at least
$\left \lceil \frac{r n}{2 N^p n} \right \rceil 
= \left \lceil \frac{r}{2 N^p} \right \rceil$
pairs $\left ( f_+, f_- \right )$ satisfy
$\left | f_+  \left ( t_{i_0} \right ) - f_- \left ( t_{i_0} \right ) \right |
\geqslant \epsilon - \frac{2 M_{\mathcal{F}}}{N}$.
Keeping in mind that 
$\epsilon - \frac{2 M_{\mathcal{F}}}{N} > \frac{4 M_{\mathcal{F}}}{N}$, 
it is easy to establish that there are less than $\frac{N^2}{2}$ different pairs
$\left ( u_1, u_2 \right ) \in \mathcal{S}^2$ such that
$\left | u_1 - u_2 \right | \geqslant \epsilon - \frac{2 M_{\mathcal{F}}}{N}$.
Thus, applying once more the pigeonhole principle, there are at least
$\left \lceil \frac{r}{N^{p+2}} \right \rceil$
pairs $\left ( f_+, f_- \right )$ such that
$\left | f_+  \left ( t_{i_0} \right ) - f_- \left ( t_{i_0} \right ) \right |
\geqslant \epsilon - \frac{2 M_{\mathcal{F}}}{N}$ and
the pair $\left ( f_+ \left ( t_{i_0} \right ),
f_- \left ( t_{i_0} \right ) \right )$ is the same.
This implies that there is a quintuplet
$\left ( i_0, s_+, s_-, \bar{\mathcal{F}}_+, \bar{\mathcal{F}}_- \right )$
such that $i_0 \in \ieg 1, n \ied$,
$\left ( s_+, s_- \right ) \in \mathcal{S}^2$
with $s_+ - s_- \geqslant \epsilon - \frac{2 M_{\mathcal{F}}}{N}$,
$\bar{\mathcal{F}}_+$ and $\bar{\mathcal{F}}_-$
are two subsets of $\bar{\mathcal{F}}$ of cardinality at least
$\left \lceil \frac{r}{N^{p+2}} \right \rceil$,
and for each $ \left ( f_+, f_- \right ) \in
\bar{\mathcal{F}}_+ \times \bar{\mathcal{F}}_-$,
the ordered pairs
$ \left ( f_+ \left ( t_{i_0} \right ), f_- \left ( t_{i_0} \right ) \right )$
and $\left ( s_+, s_- \right )$ are identical.
Obviously, any two functions in $\bar{\mathcal{F}}_+$
are $\epsilon$-separated in the metric
$d_{p, \mathbf{t}_n}$, and the same holds true for $\bar{\mathcal{F}}_-$.
So, by definition, both $\bar{\mathcal{F}}_+$ and
$\bar{\mathcal{F}}_-$
$\left ( \frac{\epsilon}{2} - \frac{3 M_{\mathcal{F}}}{N} \right )$-shatter 
at least
$\mbox{shat} \left ( 
\left \lceil \frac{r}{N^{p+2}} \right \rceil \right )$ pairs.
Neither $\bar{\mathcal{F}}_+$ nor $\bar{\mathcal{F}}_-$ 
$\left ( \frac{\epsilon}{2} - \frac{3 M_{\mathcal{F}}}{N} \right )$-shatters
any pair $\left ( s_{\mathcal{T}^q}, \mathbf{b}_q \right )$
such that $\left \{ t_{i_0} \right \} \subset s_{\mathcal{T}^q}$.
If the same pair $\left ( s_{\mathcal{T}^q}, \mathbf{b}_q \right )$ is
$\left ( \frac{\epsilon}{2} - \frac{3 M_{\mathcal{F}}}{N} \right )$-shattered
by both sets, then the pair
$\left ( s'_{\mathcal{T}^{q+1}}, \mathbf{b}'_{q+1} \right )$
where $s'_{\mathcal{T}^{q+1}} = 
\left \{ t_{i_0} \right \} \bigcup s_{\mathcal{T}^q}$
and $\mathbf{b}'_{q+1} = \left ( b_0 \; \mathbf{b}_q^T \right )^T$ is
the vector deduced from $\mathbf{b}_q$ by adding one component $b_0$
corresponding to the point $t_{i_0}$, component equal to
$\frac{2 M_{\mathcal{F}}}{N} \left \lfloor \frac{N}{4 M_{\mathcal{F}}} 
\left ( s_+ + s_- \right ) \right \rfloor$,
is $\left ( \frac{\epsilon}{2} - \frac{3 M_{\mathcal{F}}}{N} \right )$-shattered
by $\bar{\mathcal{F}}$.
Indeed,
\begin{align*}
\forall f_+ \in \bar{\mathcal{F}}_+, \;\; f_+ \left ( t_{i_0}  \right )
- \frac{2 M_{\mathcal{F}}}{N} \left \lfloor \frac{N}{4 M_{\mathcal{F}}} 
\left ( s_+ + s_- \right ) \right \rfloor
& = \; 
s_+ - \frac{2 M_{\mathcal{F}}}{N} \left \lfloor \frac{N}{4 M_{\mathcal{F}}} 
\left ( s_+ + s_- \right ) \right \rfloor \\
& \geqslant s_+ - \frac{1}{2} \left ( s_+ + s_- \right ) \\
& \geqslant \frac{1}{2} \left ( s_+ - s_- \right ) \\
& \geqslant \frac{1}{2} \left ( \epsilon 
- \frac{2 M_{\mathcal{F}}}{N} \right ) \\
& \geqslant \frac{\epsilon}{2} - \frac{3 M_{\mathcal{F}}}{N}
\end{align*}
and
\begin{align*}
\forall f_- \in \bar{\mathcal{F}}_-, \;\; f_- \left ( t_{i_0}  \right )
- \frac{2 M_{\mathcal{F}}}{N} \left \lfloor \frac{N}{4 M_{\mathcal{F}}} 
\left ( s_+ + s_- \right ) \right \rfloor
& = \; 
s_- - \frac{2 M_{\mathcal{F}}}{N} \left \lfloor \frac{N}{4 M_{\mathcal{F}}} 
\left ( s_+ + s_- \right ) \right \rfloor \\
& \leqslant 
\frac{1}{2} \left ( s_- - s_+ \right ) + \frac{2 M_{\mathcal{F}}}{N} \\
& \leqslant 
\frac{1}{2} \left ( \frac{2 M_{\mathcal{F}}}{N} - \epsilon \right ) 
+ \frac{2 M_{\mathcal{F}}}{N} \\
& \leqslant - \left ( \frac{\epsilon}{2} - \frac{3 M_{\mathcal{F}}}{N} \right ).
\end{align*}
Summarizing, for each pair $\left ( s_{\mathcal{T}^q}, \mathbf{b}_q \right )$
$\left ( \frac{\epsilon}{2} - \frac{3 M_{\mathcal{F}}}{N} \right )$-shattered 
by both $\bar{\mathcal{F}}_+$ and $\bar{\mathcal{F}}_-$,
we can exhibit by means of an injective mapping a pair
$\left ( s'_{\mathcal{T}^{q+1}}, \mathbf{b}'_{q+1} \right )$
$\left ( \frac{\epsilon}{2} - \frac{3 M_{\mathcal{F}}}{N} \right )$-shattered 
by $\bar{\mathcal{F}}$ but not by
$\bar{\mathcal{F}}_+$ or $\bar{\mathcal{F}}_-$, so that the number of pairs
$\left ( \frac{\epsilon}{2} - \frac{3 M_{\mathcal{F}}}{N} \right )$-shattered
by $\bar{\mathcal{F}}$ is superior to the sum of the number of pairs
$\left ( \frac{\epsilon}{2} - \frac{3 M_{\mathcal{F}}}{N} \right )$-shattered
by $\bar{\mathcal{F}}_+$ and the number of pairs
$\left ( \frac{\epsilon}{2} - \frac{3 M_{\mathcal{F}}}{N} \right )$-shattered
by $\bar{\mathcal{F}}_-$.
This implies that $\mbox{shat} \left ( r \right ) \geqslant
2 \cdot \mbox{shat} \left ( 
\left \lceil \frac{r}{N^{p+2}} \right \rceil \right )$.
Since it has been proved that
$\mbox{shat} \left ( 2 \right ) \geqslant 1$, by induction, for all
$u \in \mathbb{N}$ satisfying
$2 N^{\left ( p+2 \right ) u} \leqslant 
\mathcal{M} \left ( \epsilon, \mathcal{F}, d_{p, \mathbf{t}_n} \right )$,
$\mbox{shat} \left ( 2 N^{\left ( p+2 \right ) u} \right ) 
\geqslant 2^u$.
Suppose now that we can set $u = \left \lceil \log_2 \left ( K \right )
\right \rceil$. We then obtain
\begin{align*}
\mbox{shat} \left ( 2
N^{\left ( p+2 \right ) \left \lceil \log_2 \left ( K \right ) \right \rceil}
\right )
& \geqslant \; 2^{\left \lceil \log_2 \left ( K \right ) \right \rceil} \\
& > \; K-1.
\end{align*}
However, the number of pairs $\left ( s_{\mathcal{T}^q}, \mathbf{b}_q \right )$
that can be
$\left ( \frac{\epsilon}{2} - \frac{3 M_{\mathcal{F}}}{N} \right )$-shattered 
is trivially
bounded from above by the total number of those pairs, i.e., $K-1$.
We have thus established by contradiction that
\begin{equation}
\label{eq:partial-1-Sauer-Shelah-L_p}
2 N^{\left ( p+2 \right ) \left \lceil \log_2 \left ( K \right ) \right \rceil}
> \mathcal{M} \left ( \epsilon, \mathcal{F}, d_{p, \mathbf{t}_n} \right ).
\end{equation}
A well-known computation
(see for instance the proof of Corollary~3.3 in \cite{MohRosTal12})
produces the following upper bound on $K$:
\begin{equation}
\label{eq:partial-2-Sauer-Shelah-L_p}
K \leqslant \left ( \frac{e \left ( N-1 \right ) n}{d} \right )^d.
\end{equation}
Substituting (\ref{eq:partial-2-Sauer-Shelah-L_p}) into
(\ref{eq:partial-1-Sauer-Shelah-L_p}) gives:
\begin{align*}
\mathcal{M} \left ( \epsilon, \mathcal{F}, d_{p, \mathbf{t}_n} \right )
& < \;
2 N^{\left ( p+2 \right ) \left \lceil 
\log_2 \left [ \left ( \frac{e \left ( N-1 \right ) n}{d} \right )^d
\right ] \right \rceil} \\
& < \;
2 N^{\left ( p+2 \right )
\log_2 \left [ 2 \left ( \frac{e \left ( N-1 \right ) n}{d} \right )^d
\right ]} \\
& < \; 2 \left [
2 \left ( \frac{e \left ( N-1 \right ) n}{d} 
\right )^d \right ]^{\left ( p+2 \right ) \log_2 \left ( N \right )},
\end{align*}
and the last inequality is precisely Inequality~(\ref{eq:Sauer-Shelah-L_p}).
\end{proof}

\bibliography{App}
\bibliographystyle{plain}

\end{document}